\documentclass[10pt,reqno]{amsart}

\usepackage{amsmath}
\usepackage{amsfonts}
\usepackage{amssymb}
\usepackage{amsthm}
\usepackage{mathrsfs}
\usepackage{latexsym}

\usepackage{cite} 

\usepackage{esint}

\numberwithin{equation}{section}

\usepackage{graphicx,subfigure}

\usepackage{ifthen} 

\provideboolean{shownotes} 
\setboolean{shownotes}{true} 
\newcommand{\margnote}[1]{
\ifthenelse{\boolean{shownotes}}%
{\marginpar{\raggedright\tiny\texttt{#1}}}%
{}%
}

\newcommand{\hole}[1]{
\ifthenelse{\boolean{shownotes}}%
{\begin{center} \fbox{ \rule {.25cm}{0cm}
\rule[-.1cm]{0cm}{.4cm} \parbox{.85\textwidth}{\begin{center}
\texttt{#1}\end{center}} \rule {.25cm}{0cm}}\end{center}}
{}
}

\usepackage[colorlinks=true,urlcolor=blue,
citecolor=red,linkcolor=blue,linktocpage,pdfpagelabels,
bookmarksnumbered,bookmarksopen]{hyperref}

\theoremstyle{plain}

\newtheorem{lemma}{Lemma}[section]
\newtheorem{theorem}[lemma]{Theorem}
\newtheorem{proposition}[lemma]{Proposition}

\theoremstyle{definition}
\newtheorem{remark}[lemma]{Remark}
\newtheorem{definition}[lemma]{Definition}

\theoremstyle{remark}

\newcommand{\R}{\mathbb{R}}
\newcommand{\C}{\mathbb{C}}

\newcommand{\N}{\mathbb{N}}
\newcommand{\bbS}{\mathbb{S}}

\newcommand{\bw}{\boldsymbol{w}}
\newcommand{\bu}{\boldsymbol{u}}

\newcommand{\tiu}{\widetilde{u}}
\newcommand{\tiv}{\widetilde{v}}

\newcommand{\cT}{{\mathcal{T}}}

\newcommand{\cL}{{\mathcal{L}}}
\newcommand{\cD}{{\mathcal{D}}}

\newcommand{\cA}{{\mathcal{A}}}
\newcommand{\cB}{{\mathcal{B}}}

\newcommand{\vep}{\varepsilon}

\renewcommand{\Re}{\mathrm{Re}\,} 
\renewcommand{\Im}{\mathrm{Im}\,}
\newcommand{\tr}{\mathrm{tr}\,}

\newcommand{\sgn}{\mathrm{sgn}\,}

\newcommand{\taum}{\tau_{\mathrm{max}}}

\newcommand{\btau}{\overline{\tau}}

\newcommand{\ep}{\epsilon}

\newcommand{\bJ}{\mathbf{J}}
\newcommand{\bB}{\mathbf{B}}
\newcommand{\bC}{\mathbf{C}}
\newcommand{\bA}{\mathbf{A}}
\newcommand{\bI}{\mathbf{I}}
\newcommand{\bM}{\mathbf{M}}
\newcommand{\bD}{\mathbf{D}}
\newcommand{\bF}{\mathbf{F}}

\newcommand{\ess}{\sigma_\mathrm{\tiny{ess}}}
\newcommand{\ptsp}{\sigma_\mathrm{\tiny{pt}}}

\newcommand{\Ldper}{L^2_\mathrm{\tiny{per}}}
\newcommand{\Huper}{H^1_\mathrm{\tiny{per}}}

\newcommand{\Hmper}{H^m_\mathrm{\tiny{per}}}

\newcommand{\<}{\langle}
\renewcommand{\>}{\rangle}

\begin{document}

\title[Instability of small-amplitude hyperbolic periodic waves]{Spectral instability of small-amplitude periodic waves for hyperbolic non-Fickian diffusion advection models with logistic source}

\author[E. \'{A}lvarez]{Enrique \'{A}lvarez}
 
\address{{\rm (E. \'{A}lvarez)} Instituto de 
Investigaciones en Matem\'aticas Aplicadas y en Sistemas\\Universidad Nacional Aut\'onoma de 
M\'exico\\ Circuito Escolar s/n, Ciudad Universitaria, C.P. 04510\\Cd. de M\'{e}xico (Mexico)}

\email{enrique.alvarez@ciencias.unam.mx}

\author[R. Murillo]{Ricardo Murillo}

\address{{\rm (R. Murillo)} Colegio de Ciencias y Humanidades\\Universidad Nacional Aut\'onoma de 
M\'exico\\ Prol. Perif\'erico Oriente s/n Esq. Sur 24, Col. Agr\'{\i}­cola Oriental, C.P. 08500\\Cd. de M\'{e}xico (Mexico)}

\email{gaudeamus18@gmail.com}

\author[R. G. Plaza]{Ram\'on G. Plaza}

\address{{\rm (R. G. Plaza)} Instituto de 
Investigaciones en Matem\'aticas Aplicadas y en Sistemas\\Universidad Nacional Aut\'onoma de 
M\'exico\\ Circuito Escolar s/n, Ciudad Universitaria, C.P. 04510\\Cd. de M\'{e}xico (Mexico)}

\email{plaza@mym.iimas.unam.mx}

\begin{abstract}
A hyperbolic model for diffusion, nonlinear transport (or advection) and production of a scalar quantity, is considered. The model is based on a constitutive law of Cattaneo-Maxwell type expressing non-Fickian diffusion by means of a relaxation time relation. The production or source term is assumed to be of logistic type. This paper studies the existence and spectral stability properties of spatially periodic traveling wave solutions to this system. It is shown that a family of subcharacteristic periodic waves emerges from a local Hopf bifurcation around a critical value of the wave speed. These waves have bounded fundamental period and small-amplitude. In addition, it is shown that these waves are spectrally unstable as solutions to the hyperbolic system. For that purpose, it is proved that the Floquet spectrum of the linearized operator around a wave can be approximated by a linear operator whose point spectrum intersects the unstable half plane of complex numbers with positive real part. 
\end{abstract}



\maketitle

\setcounter{tocdepth}{1}



\section{Introduction}

Scalar viscous balance laws in one space dimension are equations of the form
\begin{equation}
\label{VBL}
u_t + f(u)_x = u_{xx} + g(u),
\end{equation}
where $u = u(x,t) \in \R$, $x \in \R$ and $t > 0$ denote the space and time variables, respectively, and $f = f(u)$ and $g = g(u)$ are nonlinear functions of the unknown $u$. This class of models combines nonlinear advection or transport (represented by the term $f(u)_x$), a reaction or production term (denoted by the function $g$) and diffusion or viscosity effects (modeled by the Laplace operator) into one single equation describing the dynamics of a scalar quantity $u$. For an abridged list of references on scalar viscous balance laws, see \cite{CroMas07,AlPl21,HaeSa06,Hae03}.


From a modeling point of view, equations of the form \eqref{VBL} underly two independent components. The first component is a universal balance law,
\begin{equation}
\label{UBlaw}
u_t + v_x = g(u),
\end{equation}
expressing the conservation/production of a quantity $u$ in a one dimensional spatial domain, which diffuses and gets transported according to a flux $v$ and grows/decays in a way dictated by the reaction function $g$. Equation \eqref{UBlaw} is a localized version of the global balance law in integral form,
\[
\frac{d}{dt} \int_a^b u(x,t) \, dx + v(b,t) - v(a,t) = \int_a^b g(u(x,t)) \, dx,
\]
where $(a,b) \subset \R$ is an arbitrary length element, and $u$ and $g(u)$ are interpreted as mass and production densities per unit length. Both diffusion and transport can be incorporated into the model via the function $v$, that determines the flux of the quantity $u$ across boundary elements (in the present one-dimensional setting, at $x = a,b$) of arbitrary domains. In general, $v$ depends on $u$ and its derivatives via a constitutive law. The second component is precisely the choice of this constitutive law. Fick's standard theory of diffusion states that the flux $v$ must be proportional to the gradient of $u$ (cf. \cite{Crank75,Fick95}). If we combine it with a transport mechanism and normalize it so that the diffusion coefficient is equal to one, then the Fickian constitutive law reads
\begin{equation}
\label{fick}
v = f(u) - u_x,
\end{equation}
stating that particles are transported according to the function $f$ and exhibit Fickian diffusion. Substitution of \eqref{fick} into the balance law \eqref{UBlaw} yields the parabolic equation \eqref{VBL}.

Diffusion is not, however, always modeled by a parabolic mechanism. Motivated by the unphysical infinite speed of propagation of initial disturbances associated to parabolic equations, many non-Fickian constitutive laws of diffusion have been proposed in the literature (for related discussions, see \cite{Holm93,JoPr89} and the references therein). A well-known modification has been introduced by Cattaneo \cite{Catt49,Catt58} (see also \cite{Verno58}), based on an early, pioneering intuition by J. C. Maxwell \cite{Maxw1867}, which states that there must be a required time-lag for the flux to adjust to the gradient changes, resulting into the constitutive law
\begin{equation}
\label{CattMaxw}
\tau v_t + v = f(u) - u_x.
\end{equation}
Here $\tau > 0$ is a positive parameter that represents the intrinsic relaxation time or the time scale required for the onset of diffusion within a volume element once the gradient has been established. Relation \eqref{CattMaxw} is a \emph{constitutive law of Cattaneo-Maxwell type}. In applications, the parameter $\tau$ is usually small: notice that in the limit when $\tau \to 0^+$ one (formally) recovers the Fickian constitutive law \eqref{fick}. 

In this paper we consider the constitutive law \eqref{CattMaxw} of Cattaneo-Maxwell type together with the universal balance law \eqref{UBlaw}, resulting into the hyperbolic system
\begin{equation}
\label{hypVBL}
\begin{aligned}
u_t + v_x &= g(u),\\
\tau v_t + u_x &= f(u) -v,
\end{aligned}
\end{equation}
which models nonlinear reaction and transport processes coupled with a non-Fickian diffusion mechanism for the quantity $u$. We only assume that the nonlinear advection function $f$ is of class $C^3(\R)$. Regarding the source term we suppose, for concreteness, that $g \in C^3(\R)$ and that it is of \emph{logistic type}, satisfying 
\begin{equation}
\label{A2}
	\begin{aligned}
	&g(0) = g(1) =0,\\
	&g'(0) > 0, \, g'(1) < 0,\\
	&g(u)>0 \, \textrm{ for all } \, u \in(0,1),\\
	&g(u)<0 \, \textrm{ for all } \, u \in (-\delta,0), \quad \text{some } \, \delta > 0.
	\end{aligned}
\end{equation}
Reaction functions of logistic type are used to model dynamics of populations with limited resources, which saturate into a stable equilibrium point associated to an intrinsic carrying capacity (in this case, the equilibrium state $u = 1$). They are also known as source functions of Fisher-KPP or monostable type (cf. \cite{Fis37,KPP37}).
\begin{remark}
Although it is posed as a system, \eqref{hypVBL} is essentially a scalar model for the density $u$. Indeed, one can eliminate the variable $v$ by a procedure known as \emph{Kac's trick} \cite{Hil97,Kac74}, consisting of cross-differentiation of both equations in \eqref{hypVBL}. The result is the following one-field equation for $u$,
\begin{equation}
\label{waveeq}
\tau u_{tt} - u_{xx} + (1- \tau g'(u)) u_t + f(u)_x  = g(u).
\end{equation}
Notice that this is a nonlinear wave equation with damping term and characteristic velocity given by $1- \tau g'(u)$ and by $c_*^2 = 1/\tau$, respectively.
\end{remark}

In this paper, we also assume that the physical parameter $\tau$ satisfies the relation
\begin{equation}
\label{A3}
0 < \tau < \taum := \Big( \sup_{u \in (-\delta, 1)} |g'(u)| \, \Big)^{-1},
\end{equation}
that is, the relaxation time is bounded above by a typical time scale associated to the reaction. It is to be observed that condition \eqref{A3} is tantamount to the positivity of the damping coefficient in the one-field equation \eqref{waveeq}, which is usually imposed in order to guarantee that positive initial data produce positive solutions (see \cite{LMPS16,Had99}). According to custom, we refer to \eqref{A3} as a positive damping condition.

We are interested on traveling wave solutions to system \eqref{hypVBL} of the form
\begin{equation}
\label{TWS}
(u,v)(x,t) = (U,V)(x-ct),
\end{equation}
where the profile functions $(U,V)(\cdot)$ depend only on the Galilean variable of translation, $\xi = x - ct$, and are periodic functions of its argument. The parameter $c \in\R$ is the \emph{speed of the wave} and it is assumed to satisfy 
\begin{equation}
\label{subchar}
c^2 \tau < 1.
\end{equation}
Relation \eqref{subchar} is called the \emph{subcharacteristic condition}, because it can be interpreted as the corresponding subcharacteristic relation for hyperbolic systems with relaxation \cite{L87}, in the sense that the equilibrium wave velocity cannot exceed the characteristic speed of the damped wave equation \eqref{waveeq}. 

The purpose of this paper is to study the existence and stability properties of bounded, spatially periodic traveling wave solutions to system \eqref{hypVBL}. In a recent contribution \cite{AlPl21}, we studied this problem in the parabolic case of scalar viscous balance laws of the form \eqref{VBL}. Our goal is to examine the hyperbolic variation of these diffusion-reaction-advection mechanisms under the perspective of nonlinear wave propagation. Regarding the existence of periodic traveling waves, we prove that a family of small-amplitude subcharacteristic periodic waves emerges from a Hopf bifurcation around a critical value of the wave speed. These waves have finite fundamental period and amplitude which is of order of the square root of the distance between the speed and the critical speed (see Theorem \ref{theoexist} below). The instability of the equilibrium point of the reaction is responsible for the existence of the waves and for the change of stability of the equilibrium point as the wave speed crosses the bifurcation critical value.

Next, we study the stability of these bounded periodic waves as solutions to the hyperbolic system \eqref{hypVBL}. As a customary first step, we linearize the equations around each periodic wave and analyze the resulting spectral problem. In the case of periodic waves, the linearized operator has periodic coefficients, leading to the concept of the \emph{Floquet spectrum} (see \cite{KaPro13,JMMP14,Grd1} or section \S \ref{secstab} below). We then recast the problem of locating the continuous or Floquet spectrum as a point spectral problem on an appropriate periodic space via a Bloch-type transformation. Since the waves have small amplitude, we follow previous analyses \cite{AlPl21,CheDu,KDT19} and prove that the spectrum can be approximated by the spectrum of a constant coefficient operator around the zero solution. The latter is determined by a dispersion relation that intersects the unstable half plane of complex numbers with strictly positive real part. Then we apply perturbation theory of linear operators (cf. \cite{Kat80,HiSi96}) to show that the unstable point eigenvalue of the unperturbed operator can be approximated by neighboring analytic curves of simple eigenvalues of the full problem, proving in this fashion the spectral instability of the Floquet spectra of the associated periodic wave (see Theorem \ref{theoinst} below). For that purpose, a key ingredient is to show that the perturbed problem encompasses relatively bounded perturbations. 

Hyperbolic models of reaction and diffusion of Cattaneo-Maxwell type (without advection terms) have been the subject of analytical and numerical investigations (see, e.g., \cite{DunOth86,Fedo98,Had88}), in particular from a perspective of traveling front propagation (cf. \cite{Had94,BCN14,LMPS16,LMPS19}). Since the Fick diffusion law can be retrieved from a mean field limit of a Brownian motion, it is worth mentioning that another way to interpret the Cattaneo-Maxwell law (at least in a one-dimensional setting) is through generalizations of correlated random walks (cf. \cite{Had96,Hil98}). The incorporation of transport terms, which yields the hyperbolic reaction-diffusion-convection model considered in this paper (system \eqref{hypVBL} or, equivalently, equation \eqref{waveeq}) has been also introduced and studied before in the literature by different authors (see, for instance, \cite{Hrrm12,GoOl21,GCNPC10} and some of the references therein). Up to our knowledge, however, this is the first contribution to the theory of periodic traveling waves for the system under consideration.

\subsection{Main results}
\label{mainr}
The first theorem states the existence of small-amplitude periodic waves for the hyperbolic system \eqref{hypVBL}. It is based on a local Hopf bifurcation with the wave speed as a bifurcation parameter.

\begin{theorem}[existence of small-amplitude periodic waves]
\label{theoexist}
Suppose that $f, g \in C^3(\R)$ satisfy \eqref{A2}. Then there exists a parameter value $0 < \btau < \taum$ such that for each fixed $\tau \in (0,\btau)$ we define the critical subcharacteristic speed,
\begin{equation}
\label{defc0}
c_0 = c_0(\tau) := \frac{f'(0)}{1 - \tau g'(0)},
\end{equation}
satisfying $c_0(\tau)^2 \tau < 1$, as well as the parameters
\begin{equation}
\label{defome0}
\omega_0 = \omega_0(\tau) := \sqrt{(1-c_0^2 \tau) g'(0)} \, > 0,
\end{equation}
\begin{equation}
\label{defa0}
a_0 = a_0(\tau) := \frac{(1-c_0^2\tau)^{-1}}{16} \left[ c_0 \tau g'''(0) + f'''(0) - \frac{1}{\omega_0} \big(c_0 \tau g''(0) + f''(0)\big) \big(g''(0) + c_0 f''(0)\big)\right].
\end{equation}
Assume that $a_0 \neq 0$. Then one can find $\epsilon_1 > 0$ sufficiently small such that for each $\epsilon \in (0,\epsilon_1)$ there exists a unique wave speed $c(\epsilon) = c_0 + \epsilon$ if $a_0>0$, or $c(\epsilon) = c_0 - \epsilon$ if $a_0 < 0$, and a unique (up to translations) periodic traveling wave solution to system \eqref{hypVBL} of the form 
\begin{equation}
\label{ptws}
(u,v)(x,t) = (U^\epsilon, V^\epsilon)(x - c(\epsilon)t),
\end{equation}
where the profile functions are of class $U^\epsilon, V^\epsilon \in C^3$. Moreover, the fundamental period and amplitude of the wave behave like
\begin{equation}
\label{period}
T_\epsilon = \frac{2\pi}{\omega_0} + O(\epsilon),
\end{equation}
and like
\begin{equation}
\label{amplitude}
|U^\epsilon (\xi)|, |V^\epsilon (\xi)| = O(\sqrt{\epsilon}), \qquad \text{for all } \, \xi \in \R,
\end{equation}
respectively, as $\epsilon \to 0^+$.
\end{theorem}

The second theorem refers to the spectral instability of the periodic waves found in Theorem \ref{theoexist}. 

\begin{theorem}
\label{theoinst}
Under assumption \eqref{A2} there exists $0 < {\ep}_2 \leq \ep_1$ such that every small-amplitude periodic wave $(U^\ep, V^\ep)$ from Theorem \ref{theoexist} with $0 < \ep < \ep_2$ is spectrally unstable, that is, the Floquet spectrum of the linearized operator around the wave intersects the unstable half plane $\C_+ = \{ \lambda \in \C \, : \, \Re \lambda > 0\}$.
\end{theorem}

\subsection*{Plan of the paper} Section \S \ref{secex} contains the proof of existence of small-amplitude waves upon application of the classical Andronov-Hopf bifurcation theorem. In addition, we present numerical approximations of the waves for a particular example, the hyperbolic Burgers-Fisher model. Section \S \ref{secstab} is devoted to the stability problem: we define the concept of spectral stability of a periodic wave in terms of the Floquet spectrum and show the instability result based on the perturbation theory of linear operators. Some concluding remarks and discussions can be found in section \S \ref{secconcl}.

\subsection*{On notation}
Linear operators acting on infinite-dimensional spaces are indicated with calligraphic letters (e.g., $\cL$ and $\cT$). The domain of a linear operator, $\cL : X \to Y$, with $X$, $Y$ Banach spaces, is denoted as $\cD(\cL) \subseteq X$. We denote the real and imaginary parts of a complex number $\lambda \in \C$ by $\Re\lambda$ and $\Im\lambda$, respectively, as well as complex conjugation by ${\lambda}^*$. Complex transposition is indicated by the symbol $\bA^*$, whereas simple transposition is denoted by the symbol $\bA^\top$. Standard Sobolev spaces of complex-valued functions on the real line will be denoted as $L^2(\R)$ and $H^m(\R)$, with $m \in \N$, endowed with the standard inner products and norms. We denote by $\Ldper([0,\pi])$ the Hilbert space of complex $\pi$-periodic functions in $L^2_\mathrm{\tiny{loc}}(\R)$ satisfying
$u(x + \pi) = u(x)$ a.e. in $x$ and with inner product and norm
\[
\< u, v \>_{\Ldper} = \int_0^{\pi} u(x) v(x)^* \, dx, \qquad \| u \|^2_{\Ldper} = \< u, u \>_{\Ldper}.
\]
For any $m \in \N$, the periodic Sobolev space $\Hmper([0,\pi])$ will denote the set of all functions $u \in \Ldper([0,\pi])$ with all weak derivatives up to order $m$ in $\Ldper([0,\pi])$. Their standard inner product and norm are given by $\< u,v \>_{\Hmper} = \sum_{j=0}^m \< \partial_x^j u, \partial_x^j v\>_{\Ldper}$ and $\|u\|_{\Hmper}^2 = \< u,u\>_{\Hmper}$, respectively.

\section{Existence of subcharacteristic small-amplitude periodic waves}
\label{secex}

In this section we prove the existence of small-amplitude periodic traveling waves to system \eqref{hypVBL}, where $\tau \in (0, \btau)$ for some $0 < \btau < \taum$ to be determined. We focus our attention to waves traveling with speed $c \in \R$ satisfying the subcharacteristic condition \eqref{subchar}.

\subsection{Equivalent system in the plane}

Consider traveling wave solutions to \eqref{hypVBL} of the form
\begin{equation}
\label{tws}
u(x,t) = U(x -ct), \qquad v(x,t) = V(x-ct),
\end{equation}
for some profile functions $(U,V) = (U,V)(\xi)$ of the Galilean variable of translation, $\xi = x-ct$. We look for profiles which are periodic functions of its argument with some fundamental period $T > 0$ to be determined. That is, $(U,V)(\xi + T) = (U,V)(\xi)$ for all $\xi \in \R$. Substitution of \eqref{tws} into \eqref{hypVBL} yields the following family of planar systems parametrized by $c$:
\begin{equation}
\label{ODE}
\begin{aligned}
-c U_\xi + V_\xi &= g(U),\\
U_\xi - c \tau V_\xi &= f(U) - V
\end{aligned}
\end{equation}
or, equivalently,
\[
\begin{pmatrix}
-c & 1 \\ 1 & -c \tau
\end{pmatrix} \begin{pmatrix}
U_\xi \\ V_\xi 
\end{pmatrix} = \begin{pmatrix}
g(U) \\ f(U) - V 
\end{pmatrix}.
\]
Assuming that the subcharacteristic condition \eqref{subchar} holds, system \eqref{ODE} can be recast as 
\begin{equation}
\label{ODE2}
\begin{aligned}
U_\xi &= F(U,V,c,\tau),\\
V_\xi &= G(U,V,c,\tau),
\end{aligned}
\end{equation}
where,
\[
\begin{aligned}
F(U,V,c,\tau) &:= (1 - c^2 \tau)^{-1} \big( c \tau g(U) + f(U) - V\big), \\
G(U,V,c,\tau) &:= (1 - c^2 \tau)^{-1} \big( g(U) + c(f(U) - V)\big).
\end{aligned}
\]
Under the regularity assumption $f,g \in C^3(\R)$, it is clear that $F, G \in C^3$ as functions of $(U,V)$. From inspection, it is also easy to verify that the only equilibrium points in the $(U,V)$-plane of system \eqref{ODE2} are $P_0 = (0, f(0))$ and $P_1 = (1, f(1))$. Let us compute the linearization at an equilibrium point. The Jacobian matrix of system \eqref{ODE2} is given by
\[
\bJ(U,V,c,\tau) = \begin{pmatrix}
F_U & F_V \\ G_U & G_V
\end{pmatrix} = (1 - c^2 \tau)^{-1} \begin{pmatrix}
c \tau g'(U) + f'(U) & -1 \\ g'(U) + c f'(U) & -c
\end{pmatrix}.
\]

We now pay attention to the equilibrium point $P_0 = (0,f(0))$, corresponding to the unstable equilibrium of the reaction ($u = 0$ with $g(0) = 0$, $g'(0) > 0$) and the value of the nonlinear flux at that point ($v = f(0)$). The linearization of \eqref{ODE2} at $P_0$ is given by
\[
\bJ_0 := \bJ(0,f(0), c,\tau) = (1 - c^2 \tau)^{-1} \begin{pmatrix}
c \tau g'(0) + f'(0) & -1 \\ g'(0) + c f'(0) & -c
\end{pmatrix}.
\]
The eigenvalues of $\bJ_0$ are the $\zeta$-roots of
\[
\det \begin{pmatrix}
c \tau g'(0) + f'(0)-\zeta & -1 \\ g'(0) + c f'(0) & -c - \zeta
\end{pmatrix} = \zeta^2 + (c - f'(0) - c \tau g'(0)) \zeta + (1-c^2\tau) g'(0) = 0,
\]
namely,
\[
\zeta_0^\pm(c,\tau) := \tfrac{1}{2} \big(f'(0) + c\tau g'(0) - c\big) \, \pm \, \tfrac{1}{2} \Big( (f'(0) + c\tau g'(0) - c)^2 - 4 (1-c^2\tau) g'(0) \Big)^{1/2}.
\]

Notice that, in view of the subcharacteristic condition \eqref{subchar} and of assumption \eqref{A2}, the eigenvalues are purely imaginary when they are evaluated at the following critical value of the speed \eqref{defc0},
\[
c_0 = \frac{f'(0)}{1 - \tau g'(0)}.
\]
Here $1 - \tau g'(0) > 0$ because of the positive damping condition \eqref{A3}. To ensure that $c_0$ is a subcharacteristic speed we further restrict the parameter domain for $\tau$. Clearly, $c_0^2 < 1/\tau$ if and only if
\[
\psi(\tau) := (1-\tau g'(0))^2 - \tau f'(0)^2 > 0.
\]
Notice that $\psi(0) > 0$. It takes a straightforward calculation to verify that, under our assumptions, the roots of this second order polynomial in $\tau$ are both real and positive. The smallest of these roots is
\[
\tau_1 := \frac{1}{2g'(0)^2} \left( f'(0)^2 + 2 g'(0) - \sqrt{(f'(0)^2 + 2 g'(0))^2 - 4g'(0)^2} \right) > 0.
\]
Therefore, we define
\begin{equation}
\label{defbtau}
\btau := \min \left\{ \taum, \tau_1 \right\} > 0,
\end{equation}
and we specialize the analysis to parameter values $\tau \in (0,\btau)$, ensuring both the positive damping, \eqref{A3}, and subcharacteristic, \eqref{subchar}, conditions.


The existence of small-amplitude periodic traveling waves for systems of the form \eqref{hypVBL} is a direct consequence of the classical Andronov-Hopf's theorem in the plane \cite{Andro29,Hop42} (for the precise statement that we apply here, see Theorem 2.1 in \cite{AlPl21}; the reader is also referred to its different versions in \cite{GuHo83,HaKo91,Kuz982e}).

Let us gather all the necessary ingredients prior to the application of this classical local bifurcation result. First, notice that for each fixed $\tau \in (0,\btau)$ the eigenvalues of $\bJ_0$ can be written as
\[
\zeta_0^\pm (c) = \alpha(c) \pm i \beta(c),
\]
where
\[
\begin{aligned}
\alpha(c) &:= \frac{1}{2} \big( f'(0) + c( \tau g'(0) - 1) \big),\\
\beta(c) &:= \frac{1}{2} \sqrt{4 (1-c^2 \tau) g'(0) - (f'(0) + c\tau g'(0) - c)^2 \, },
\end{aligned}
\]
are both defined for $c \approx c_0$. Note that $\beta(c) \in \R$ for $c$ near $c_0$ because $g'(0) > 0$. Observe as well that $\alpha(c_0) = 0$ and the point $P_0 = (0, f(0))$ is a center for system \eqref{ODE2} when $c = c_0$, with eigenvalues $\zeta_0^\pm(c_0) = \pm i \beta(c_0)$. Hence, the bifurcation parameter is the speed $c$ and the critical value for which a bifurcation occurs is $c = c_0$.  Let us now define
\[
\omega_0 := \beta(c_0) = \sqrt{(1-c_0^2\tau) g'(0)} \, > 0,
\]
and compute
\[
G_U(0,f(0),c_0,\tau) = (1 - c_0^2 \tau)^{-1} \big( g'(0) + c_0 f'(0) \big) =  (1 - c_0^2 \tau)^{-1} \Big( g'(0) + \frac{f'(0)^2}{1 - \tau g'(0)} \Big) > 0,
\]
for all $\tau \in (0, \btau)$ in view of $1 > \tau g'(0)$, $g'(0) > 0$ and $1 > c_0^2 \tau$. Therefore,
\begin{equation}
\label{samesign}
\sgn (\omega_0) = \sgn ({G_U}_{|(P_0,c_0)}) = 1.
\end{equation}
Moreover, it is clear that
\begin{equation}
\label{defd0}
d_0 : = \frac{d\alpha}{dc}(c_0) = \tfrac{1}{2}(\tau g'(0) -1) < 0, \qquad \forall \, \tau \in (0,\btau). 
\end{equation}

Finally, let us recall the expression for the first Lyapunov coefficient in the case when a planar system, like \eqref{ODE2}, is not in normal form (see, for instance, \cite{GuHo83}, p. 152, or \cite{Verdu05}, sect. 2):
\begin{equation}
\label{formlyapexp}
\begin{aligned}
 a_0 &:= \frac{1}{16} \big( F_{UUU} + F_{UVV} + G_{UUV} + G_{VVV}\big) +\\ &\; \; + \frac{1}{16 \omega_0} 
\big((F_{UU} + F_{VV})F_{UV}  - (G_{UU} + G_{VV})G_{UV} - F_{UU}G_{UU} + F_{VV}G_{VV}\big).
\end{aligned}
\end{equation}
Here the partial derivatives of $F$ and $G$ are evaluated at the equilibrium point $P_0$ and at the critical value $c = c_0$. Since $F_V = - (1-c^2 \tau)^{-1}$ and $G_V = - c (1-c^2 \tau)^{-1}$ are constant, all the higher order derivatives with respect to $V$ vanish and expression \eqref{formlyapexp} reduces to
\[
a_0 = \frac{1}{16} \left( F_{UUU} - \frac{1}{\omega_0} F_{UU} G_{UU}\right)_{|(P_0,c_0)}.
\]
Upon computation of the derivatives,
\[
\begin{aligned}
(F_{UU})_{|(P_0,c_0)} &= (1 - c_0^2 \tau)^{-1} (c_0 \tau g''(0) + f''(0)),\\
(F_{UUU})_{|(P_0,c_0)} &= (1 - c_0^2 \tau)^{-1} (c_0 \tau g'''(0) + f'''(0)),\\
(G_{UU})_{|(P_0,c_0)} &= (1 - c_0^2 \tau)^{-1} (g''(0) + c_0 f''(0)),
\end{aligned}
\]
we obtain the expression,
\[
a_0 = \frac{(1-c_0^2\tau)^{-1}}{16} \left[ c_0 \tau g'''(0) + f'''(0) - \frac{1}{\omega_0} \big(c_0 \tau g''(0) + f''(0)\big) \big(g''(0) + c_0 f''(0)\big)\right],
\]
yielding \eqref{defa0} for each $\tau \in (0, \btau)$. We now have all the elements to prove the main existence theorem.

\subsection{Proof of Theorem \ref{theoexist}}

Follows from the verification of the hypotheses of Andronov-Hopf's theorem. First, note that the non-hyperbolicity condition is satisfied because eigenvalues of $\bJ_0$ are purely imaginary, $\zeta_0^\pm(c_0) = \pm i \beta(c_0)$; $c = c_0$ is the critical value where the bifurcation occurs, and \eqref{samesign} holds. The transversality condition is also satisfied in view of \eqref{defd0}. Finally, by hypothesis, the first Lyapunov coefficient is different from zero, $a_0 \neq 0$, and the genericity condition is also verified.

Therefore, we apply Andronov-Hopf's bifurcation theorem to system \eqref{ODE2} with fixed $\tau \in (0,\btau)$, where $\btau$ is defined in \eqref{defbtau} and where $c$ plays the role of the bifurcation parameter, to conclude the existence of $\epsilon_0 > 0$ such that a unique family of periodic orbits bifurcate from $P_0$ into the region $c \in (c_0, c_0 + \epsilon_0)$ if $a_0 d_0 < 0$, or into the region $c \in (c_0 - \epsilon_0, c_0)$ if $a_0 d_0 > 0$. In view of \eqref{defd0}, $d_0 < 0$ and the equilibrium point $P_0$ is stable for $c > c_0$ and unstable for $c < c_0$. Hence, these periodic orbits are unstable\footnote{see Remark \ref{remstab} below.} for $c > c_0$ (subcritical Hopf bifurcation) and stable for $c < c_0$ (supercritical Hopf bifurcation). The amplitude of these periodic orbits grows like $O(\sqrt{|c-c_0|})$ and their periods satisfy $T = 2\pi/|\omega_0| + O(|c-c_0|)$ (see Theorem 3.1, p. 65 in \cite{MaMcC76}, or section 8.2 in \cite{Strg15}). Notice that $\tau \in (0,\btau)$ and therefore the critical wave speed $c_0$ in \eqref{defc0} is subcharacteristic. By continuity, there exists $\widetilde{\epsilon} > 0$ such that if $|c-c_0| < \widetilde{\epsilon}$ then $c$ remains subcharacteristic, that is, $c^2 \tau < 1$. Thus, we select $\epsilon_1 := \min \{\epsilon_0, \widetilde{\epsilon}\} > 0$. 

Now, since the wave speeds, 
\begin{equation}
\label{defce}
c(\ep) := \begin{cases}
c_0 + \epsilon, & \text{if } a_0 > 0,\\
c_0 - \epsilon, & \text{if } a_0 < 0,
\end{cases} \qquad \epsilon \in [0,\epsilon_1)
\end{equation}
are subcharacteristic with $c(\epsilon)^2 \tau < 1$, then the nonlinear system \eqref{ODE2} is equivalent to the original system \eqref{ODE} and the orbits define periodic traveling wave solutions to the hyperbolic model \eqref{hypVBL} of the form \eqref{tws}, whose fundamental period and amplitude satisfy \eqref{period} and \eqref{amplitude}. The theorem is proved.
\qed

\begin{remark}
\label{remstab}
The notion of a “stable'' periodic orbit that we mention in the proof of Theorem \ref{theoexist} refers to the standard concept from dynamical systems theory: the orbit is stable as a solution to system \eqref{ODE2} for a specific (and constant) value of $c$ if any other nearby solution (to the system with the same $c$) tends to the orbit under consideration. This notion is completely unrelated to the concept of \emph{spectrally stable periodic wave} from Definition \ref{defspectstab} below, which refers to the dynamical stability of the traveling wave as a solution to the evolution system of PDEs. In \cite{AlPl21}, section 6.3, the reader can find an example of a periodic wave which is stable as a periodic orbit of the associated dynamical system but spectrally unstable as solution to the PDE.
\end{remark}

\subsection{Example: the hyperbolic Burgers-Fisher model}

Let us examine the case of the hyperbolic Burgers-Fisher model, namely,
\begin{equation}
\label{hypBF}
\begin{aligned}
u_t + v_x &= u(1-u),\\
\tau v_t + u_x &= \tfrac{1}{2}u^2 -v.
\end{aligned}
\end{equation}
for which the reaction and nonlinear advection flux are given by the classical logistic parabolic profile (also known as Fisher-KPP reaction\cite{Fis37,KPP37}),
\begin{equation}
\label{Fisher}
g(u) = u(1-u).
\end{equation}
and by the Burgers' flux function \cite{Bur48,La57}, 
\begin{equation}
\label{Burgers}
f(u) = \tfrac{1}{2}u^2,
\end{equation}
respectively. The Burgers' flux plays a significant role in the theory of scalar conservation laws (cf. \cite{Da4e,La73}) as it is the paradigm of a convex (genuinely nonlinear) mode. The logistic (concave) parabolic profile is a classical choice for a source term satisfying \eqref{A2} and goes back to the work of Verhulst \cite{Verh1838} in population dynamics. Hence, the hyperbolic Burgers-Fisher system \eqref{hypBF} is perhaps the simplest example of a hyperbolic reaction-diffusion-advection model of the from \eqref{hypVBL}.

In this case it is clear that, for all $u \in \R$,
\[
\begin{aligned}
f'(u) &= u, &\quad f''(u) &= 1, &\quad f'''(u) &= 0,\\
g'(u) &= 1-2u, &\quad g''(u) &= -2, &\quad g'''(u) &=0,
\end{aligned}
\]
so that, upon substitution,
\[
c_0 = \frac{f'(0)}{1- \tau g'(0)} = 0, \quad \omega_0 = \sqrt{(1-c_0^2 \tau)g'(0)}  = 1, \quad a_0 = - \frac{1}{16} f''(0) g''(0) = \frac{1}{8} > 0. 
\]
Moreover, it is easy to verify that $\taum = \btau = 1$ and the hypotheses of Theorem \ref{theoexist} are satisfied. Thus, we have the following 
\begin{proposition}
Fix $\tau \in (0,1)$. Hence there exists $\epsilon_1 > 0$ sufficiently small (depending on $\tau$) such that for each $c(\epsilon) := \epsilon > 0$, $0 < \epsilon < \epsilon_1$, there is a unique (up to translations) periodic traveling wave solution to the hyperbolic Burgers-Fisher model \eqref{hypBF} of the form \eqref{tws}, traveling with speed $c(\epsilon)$, with amplitude of order $O(\sqrt{\epsilon})$ and fundamental period of order $T = 2\pi + O(\epsilon)$ as $\epsilon \to 0^+$.
\end{proposition}

Figures \ref{figT02BF} and \ref{figT09BF} illustrate the emergence of small amplitude periodic traveling waves from a Hopf bifurcation for the hyperbolic Burgers-Fisher system. Both figures are represented in the same scale for comparison purposes. 

\begin{figure}[h]
\begin{center}
\subfigure[$c = -0.01$]{\label{figT02BFneg}\includegraphics[scale=.45, clip=true]{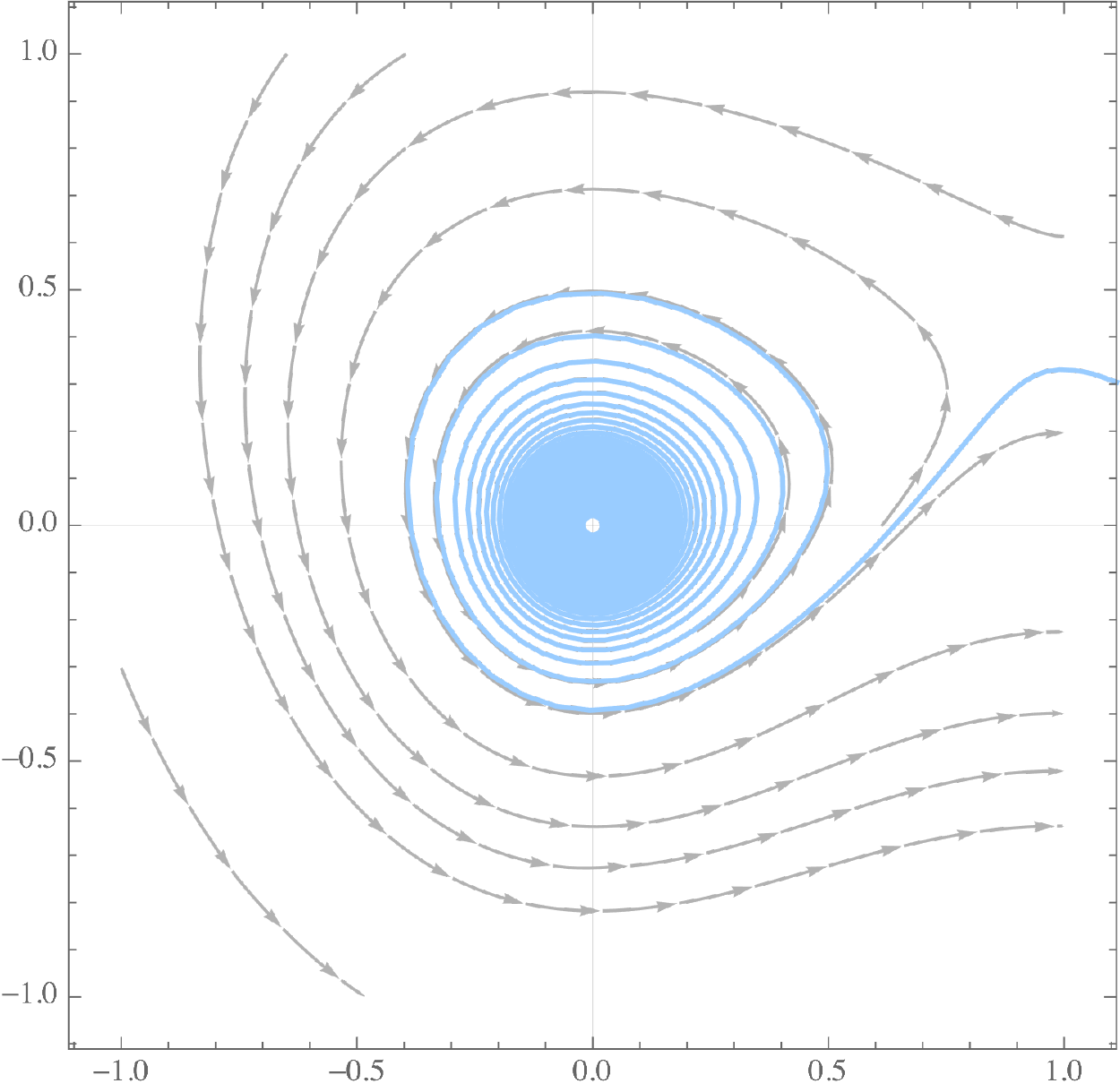}}
\subfigure[$c=0$]{\label{figT02BF0}\includegraphics[scale=.45, clip=true]{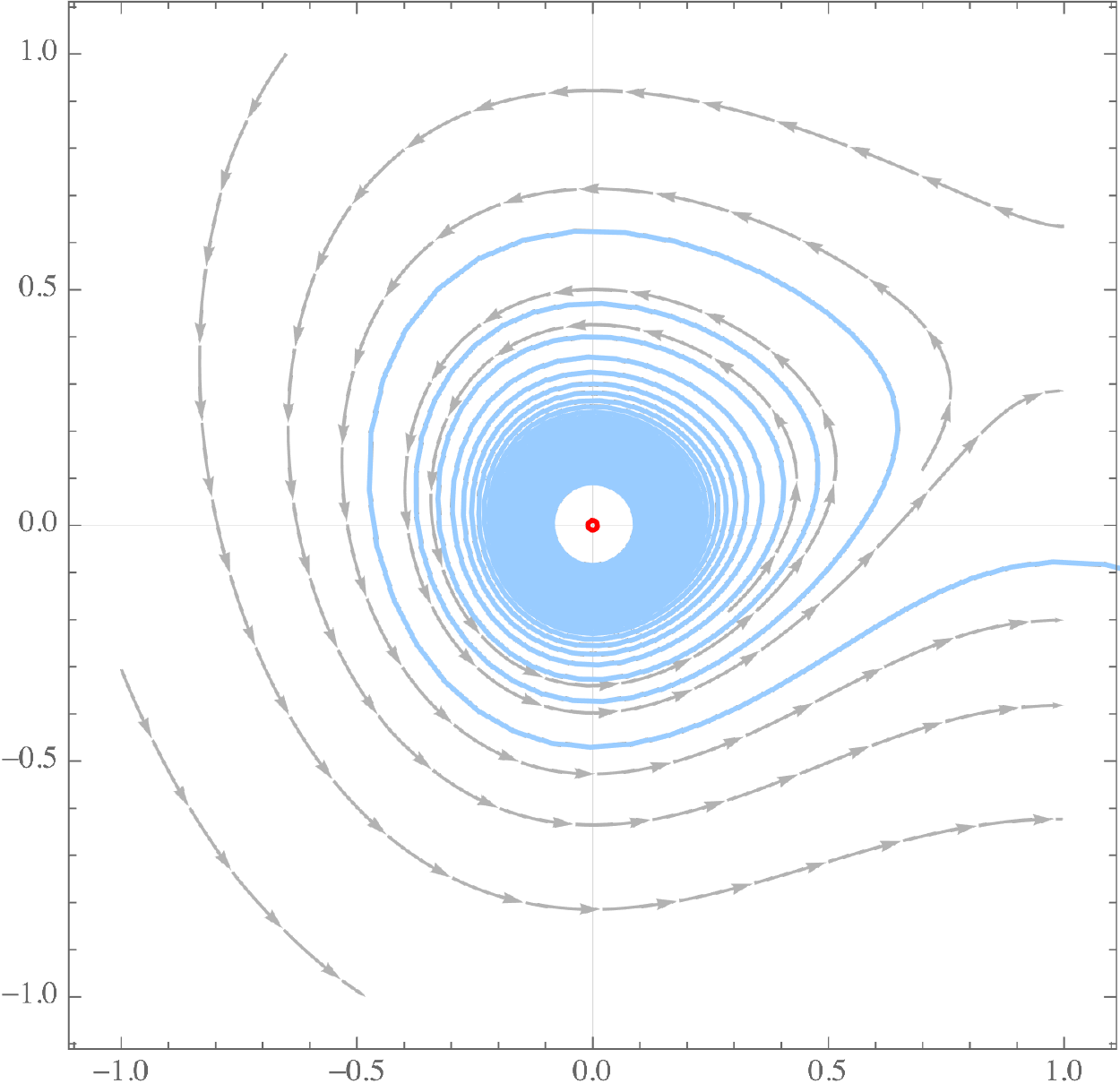}}
\subfigure[$c= 0.01$]{\label{figT02BFpos}\includegraphics[scale=.45, clip=true]{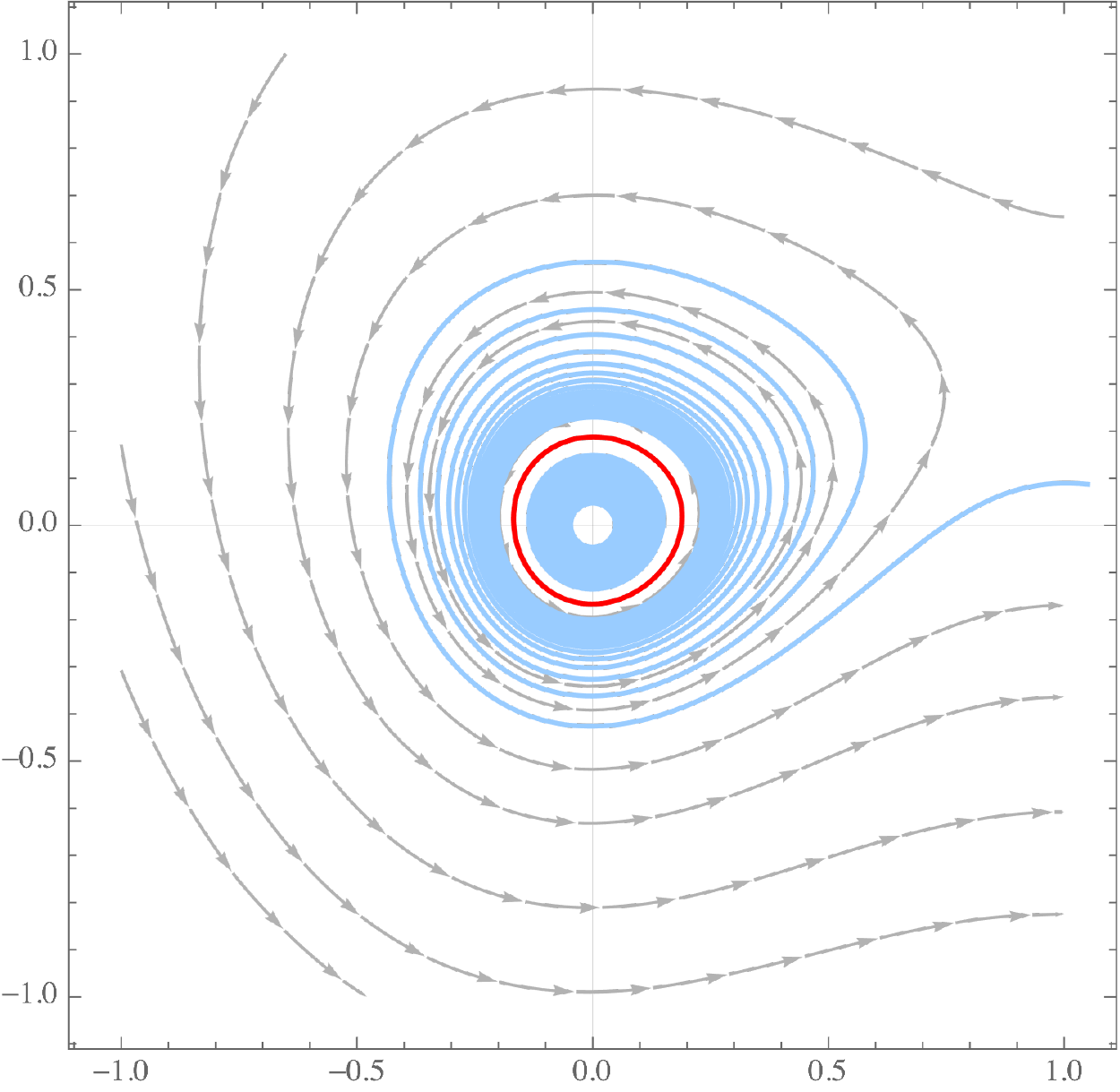}}
\subfigure[$U = U(\xi)$]{\label{figT02BFwave}\includegraphics[scale=.45, clip=true]{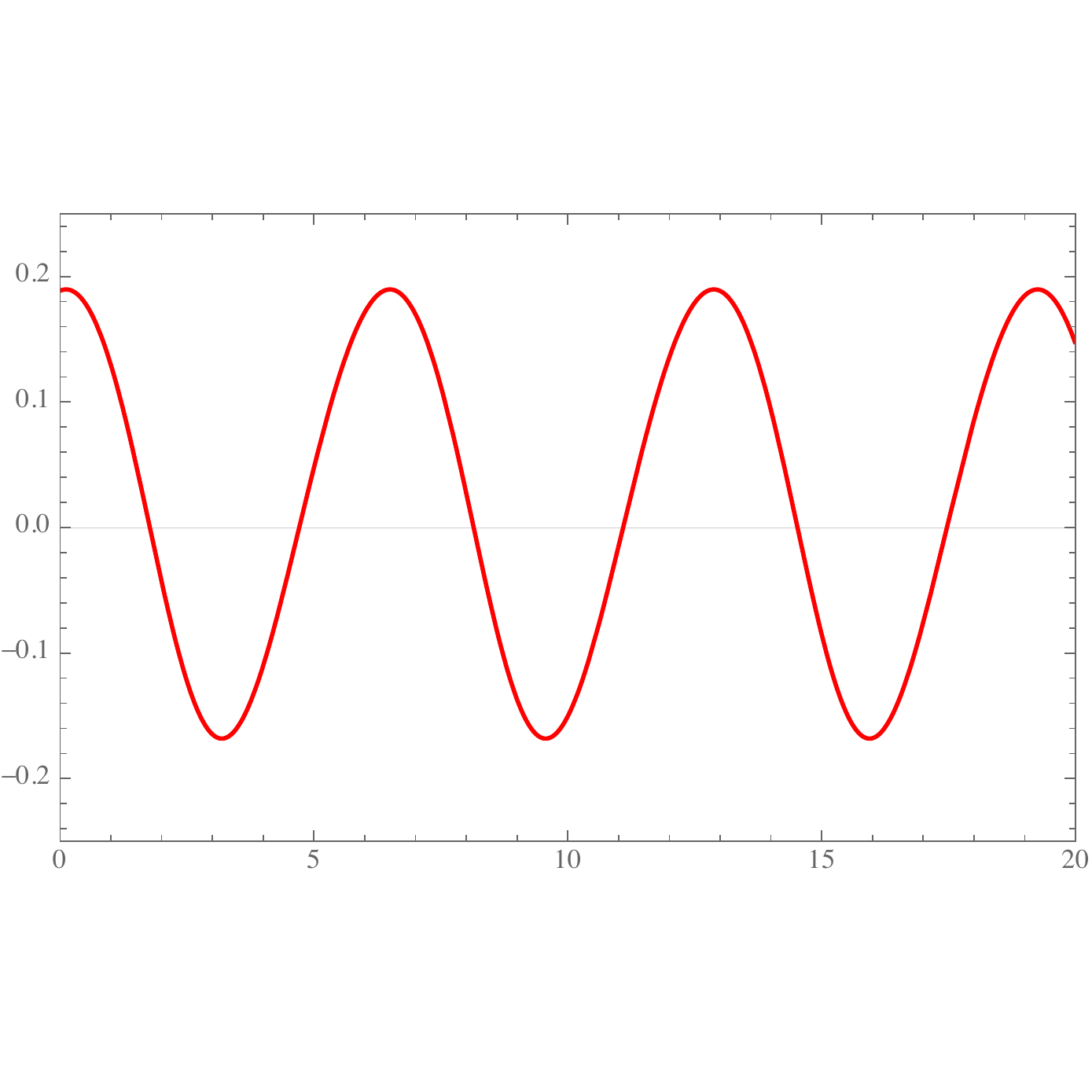}}
\end{center}
\caption{\small{Emergence of small-amplitude periodic waves for the hyperbolic Burgers-Fisher model \eqref{hypBF} with relaxation time $\tau = 0.2$. Panel (a) shows the phase portrait in the $(U,V)$ plane of system \eqref{ODE2} for the speed value $c = -0.01$. Numerical solutions are shown in light blue color. Panel (b) shows the case when $c = c_0 = 0$, the parameter value where a subcritical Hopf bifurcation occurs. Panel (c) shows the case where $c = 0.01$: the orbit in red is a numerical approximation of the unique small amplitude periodic wave for this speed value. Panel (d) shows the graph (in red) of the $U$-component of the approximated periodic wave (vertical axis) as a function of the Galilean variable $\xi = x -ct$ (horizontal axis) (color online).}}\label{figT02BF}
\end{figure}

\begin{figure}[h]
\begin{center}
\subfigure[$c = -0.01$]{\label{figT09BFneg}\includegraphics[scale=.45, clip=true]{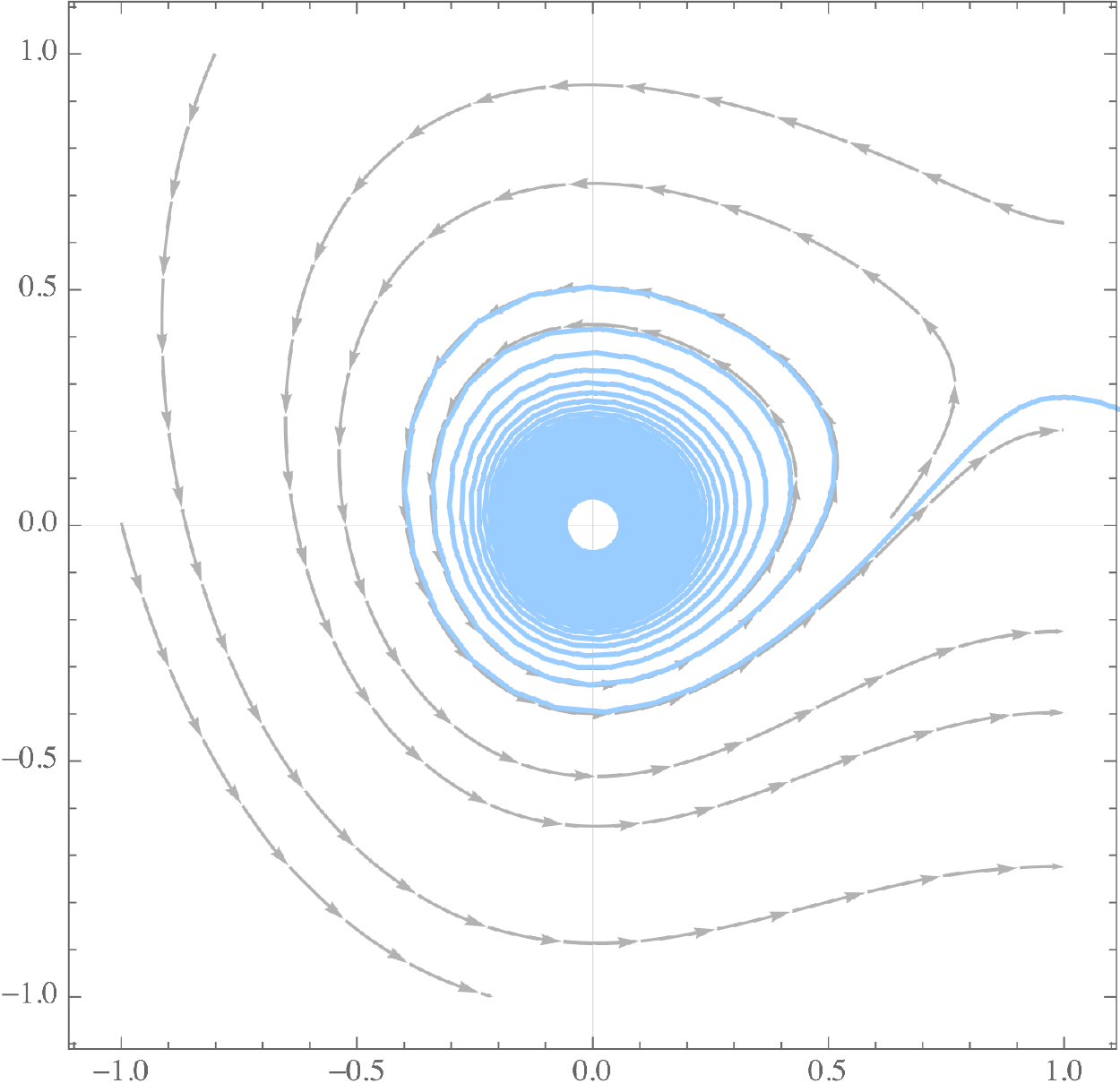}}
\subfigure[$c=0$]{\label{figT09BF0}\includegraphics[scale=.45, clip=true]{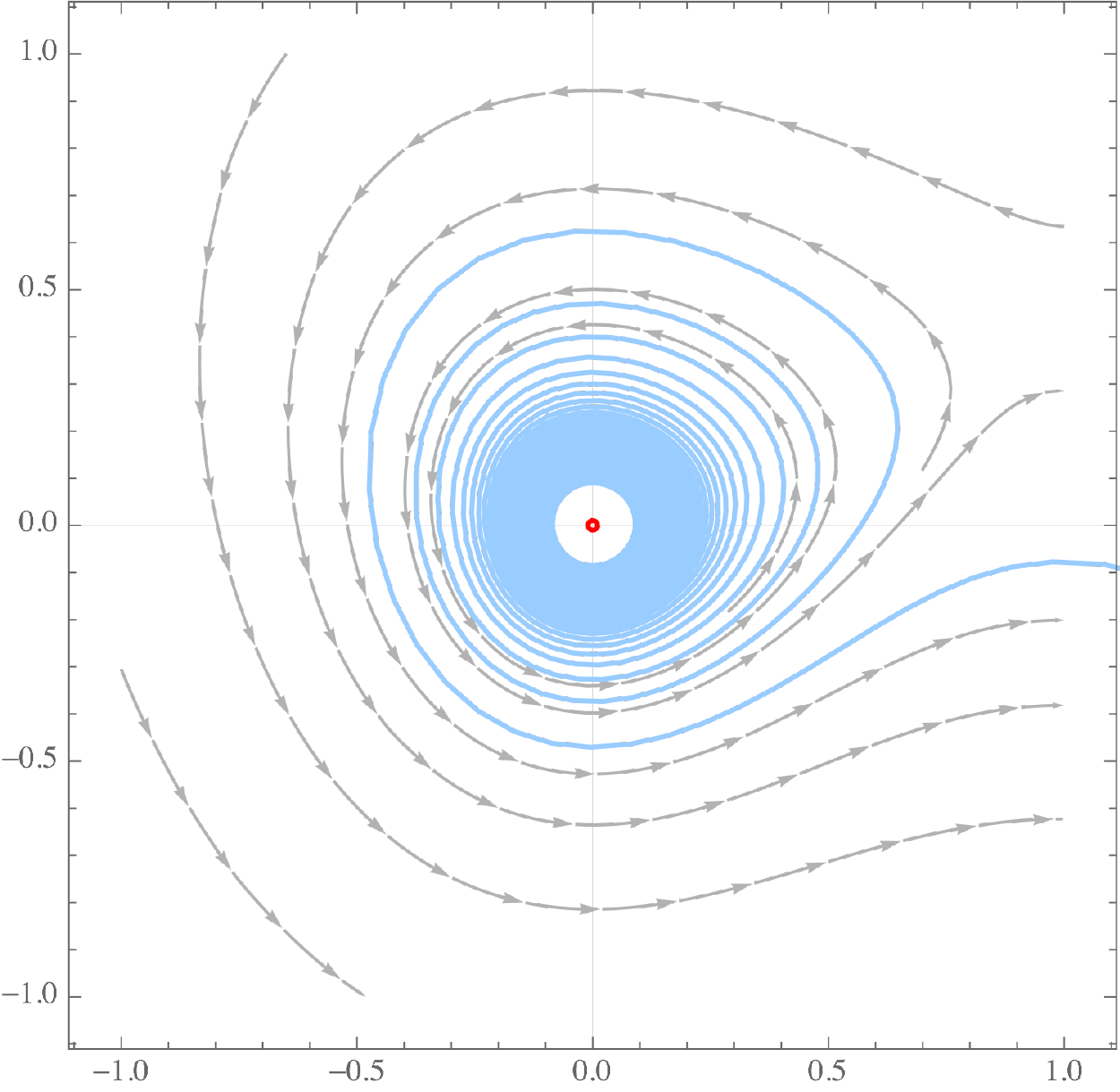}}
\subfigure[$c= 0.01$]{\label{figT09BFpos}\includegraphics[scale=.45, clip=true]{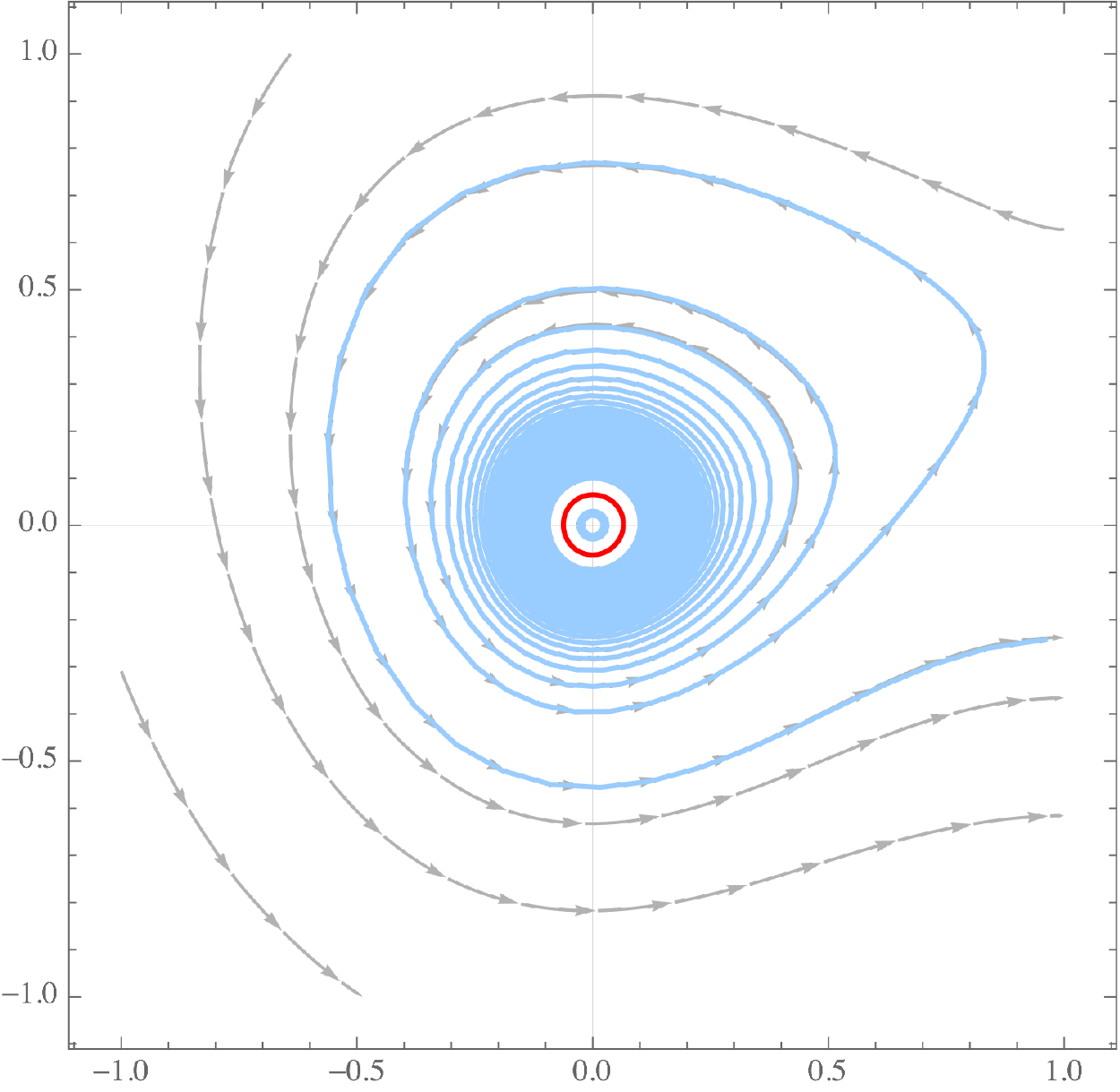}}
\subfigure[$U = U(\xi)$]{\label{figT09BFwave}\includegraphics[scale=.45, clip=true]{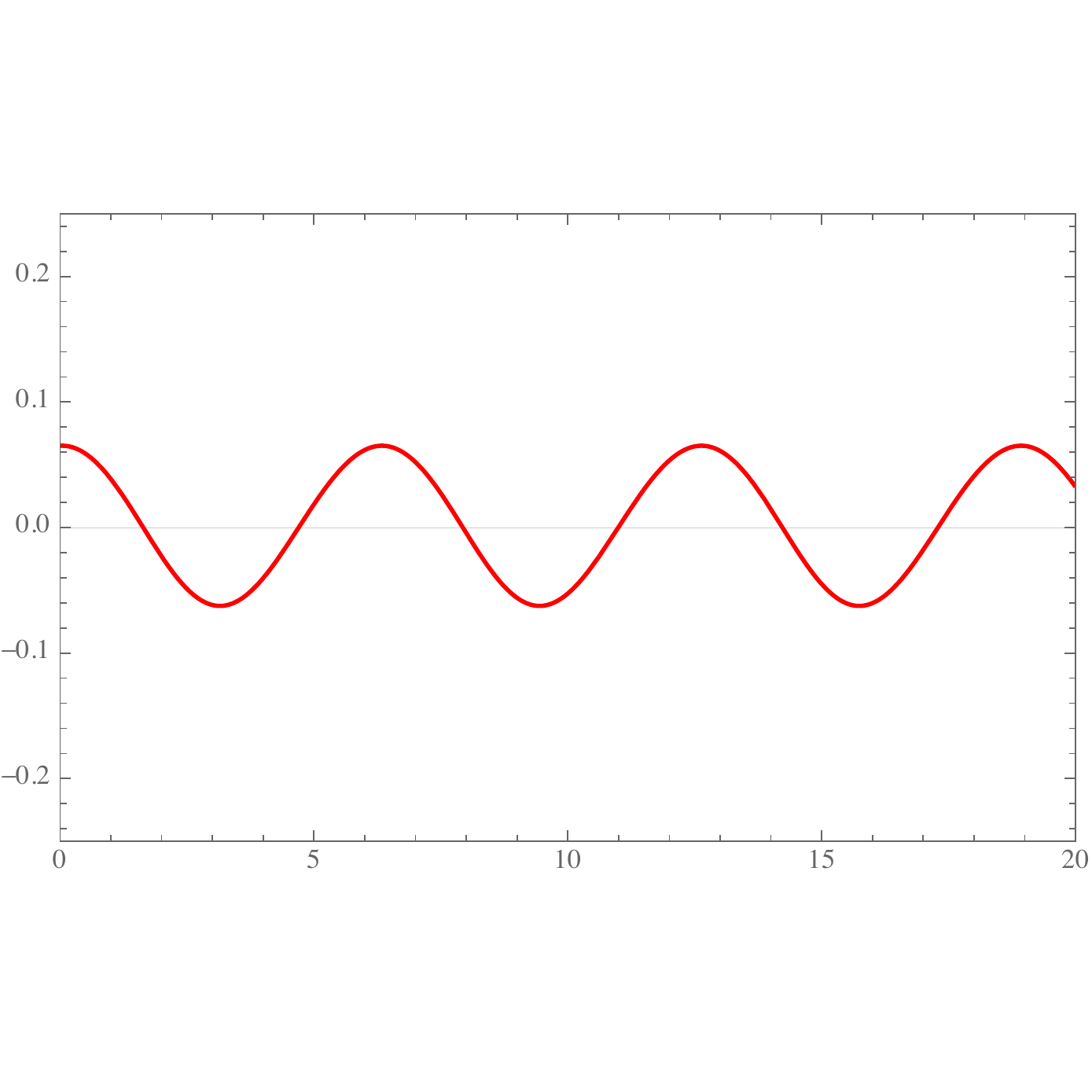}}
\end{center}
\caption{\small{Emergence of small-amplitude periodic waves for the hyperbolic Burgers-Fisher model \eqref{hypBF} with relaxation time $\tau = 0.9$. Panel (a) shows the phase portrait in the $(U,V)$ plane of system \eqref{ODE2} for the speed value $c = -0.01$. Numerical solutions are shown in light blue color. Panel (b) shows the case when $c = c_0 = 0$, the parameter value where a subcritical Hopf bifurcation occurs. Panel (c) shows the case where $c = 0.01$: the orbit in red is a numerical approximation of the unique small amplitude periodic wave for this speed value. Panel (d) shows the graph (in red) of the $U$-component of the approximated periodic wave (vertical axis) as a function of the Galilean variable $\xi = x -ct$ (horizontal axis) (color online).}}\label{figT09BF}
\end{figure}

In Figure \ref{figT02BF} we consider system \eqref{hypBF} for the parameter value $\tau = 0.2$. Figure \ref{figT02BFneg} shows the phase portrait of system \eqref{ODE2} with speed value $c = -0.01$ (that is, below the critical speed): the origin is a repulsive node and all nearby solutions (in light blue; color online) move away from it. Figure \ref{figT02BF0} shows the phase portrait of \eqref{ODE2} for the critical value $c = c_0 = 0$, where the subcritical Hopf bifurcation occurs. Figure \ref{figT02BFpos} depicts the phase portrait for the speed value $c = 0.01$. The closed orbit in red color is a numerical approximation of the (unique) small-amplitude periodic orbit for this value of the speed: the origin is now an attractive node and all nearby solutions inside the periodic orbit approach zero, whereas solutions outside move away from it. Panel \ref{figT02BFwave} shows the periodic profile, $U = U(\cdot)$ (in red), as a function of the Galilean variable $\xi = x-ct$, computed numerically for $c = 0.01$. Its fundamental period is approximately $T \approx 2\pi$, according to formula \eqref{period}.

Figure \ref{figT09BF} shows the emergence of small-amplitude periodic waves from a Hopf bifurcation in the same fashion as in Figure \ref{figT02BF}, but for the relaxation time $\tau = 0.9$. Notice that the amplitude of the waves seems to decrease as $\tau$ increases. We do not know whether this is a general behavior. The existence of periodic orbits in the limit when $\tau$ tends to the threshold value $\btau$, where the system becomes singular as it approaches a characteristic speed, or even beyond the subcharacteristic regime, is a question that deserves further investigations.

\section{Spectral instability}

\label{secstab}

In this section we prove that the small-amplitude periodic waves emerging from a Hopf bifurcation are spectrally unstable, that is, that the Floquet spectrum of the linearized operator around the periodic wave intersects the unstable half plane of complex numbers with positive real part.

\subsection{Perturbation equations and the stability problem}

According to Theorem \ref{theoexist}, let us fix a value $\tau \in (0, \btau)$ and consider the periodic traveling wave solution to \eqref{hypVBL}, $(u,v)(x,t) = (U^\ep, V^\ep)(x-c(\ep)t)$, with speed given by \eqref{defce} for each $\ep \in (0, \ep_1)$. With a slight abuse of notation let us rescale the space variable as $x \mapsto x - c(\ep)t$ in order to transform system \eqref{hypVBL} into
\begin{equation}
\label{hVBLp}
\begin{aligned}
-cu_x + u_t + v_x &= g(u),\\
-c \tau v_x + \tau v_t + u_x &= f(u)-v,
\end{aligned}
\end{equation} 
for which the periodic waves constitute stationary solutions, $(U^\ep, V^\ep) = (U^\ep, V^\ep)(x)$. Consider a solution to \eqref{hVBLp} of the form $(U^\ep,V^\ep)(x) + (\tiu,\tiv)(x,t)$ where $\tiu$ and $\tiv$ denote perturbations of the wave.  Upon substitution into \eqref{hVBLp} and linearization around the wave, we obtain the following linear system for the perturbation variables,
\begin{equation}
\label{lin}
\begin{aligned}
\tiu_t &= c \tiu_x - \tiv_x + g'(U^\ep) \tiu,\\
\tau \tiv_t &= c\tau \tiv_x - \tiu_x - \tiv + f'(U^\ep) \tiu.
\end{aligned}
\end{equation}

If we consider perturbations of the form $(\tiu,\tiv) = e^{\lambda t} (u,v)$, where $\lambda \in \C$ and $(u,v)^\top \in L^2(\R) \times L^2(\R)$ we arrive at the following spectral problem
\begin{equation}
\label{spectp}
\lambda \bw = \cL^\ep \bw,
\end{equation}
where $\bw = (u,v)^\top$ and
\[
\cL^\ep : \cD(\cL^\ep) \subset L^2(\R) \times L^2(\R) \rightarrow L^2(\R) \times L^2(\R),
\]
is a closed, densely defined operator with domain $\cD(\cL^\ep) = H^1(\R) \times H^1(\R)$ and determined by
\begin{equation}
\cL^\ep := \bB^{-1} \big( \bA \partial_x + \bC(x) \big),
\end{equation}
where
\[
\bA := \begin{pmatrix} c & -1 \\ -1 & c\tau
\end{pmatrix}, \qquad \bB := \begin{pmatrix}1 & 0 \\ 0 & \tau
\end{pmatrix}, \qquad \bC(x) := \begin{pmatrix} b(x) & 0 \\ a(x) & -1
\end{pmatrix}, 
\]
and,
\[
a(x) := f'(U^\ep(x)), \qquad b(x) := g'(U^\ep(x)). 
\]
Notice that since the profile function $(U^\ep, V^\ep)$ is periodic with period $T_\ep$, the coefficients of the operator $\cL^\ep$ are periodic with fundamental period $T_\ep$.

The $L^2$-resolvent set of $\cL^\ep$, denoted as $\rho(\cL^\ep)$, is defined as the set of complex numbers $\lambda \in \C$ such that $\cL^\ep - \lambda$ is invertible and $(\cL^\ep - \lambda)^{-1}$ is a bounded operator. The complement of $\rho(\cL^\ep)$ is what we call the $L^2$-spectrum of $\cL^\ep$ and we denote it as $\sigma(\cL^\ep) = \C \backslash \rho(\cL^\ep)$.

\begin{definition}[spectral stability]
\label{defspectstab} 
We say that the periodic traveling wave solution, $(U^\ep,V^\ep)$, to system \eqref{hypVBL} is \emph{spectrally stable} if the $L^2$-spectrum of the linearized operator around the wave satisfies
\[
\sigma(\cL^\ep) \subset \{\lambda \in \C \, : \, \Re \lambda \leq 0\}.
\]
Otherwise, we say that it is \emph{spectrally unstable}.
\end{definition}

\begin{remark}
Any $\lambda \in \sigma(\cL^\ep)$ is said to belong to the point spectrum, $\ptsp(\cL^\ep)$, if $\cL^\ep - \lambda$ is a Fredholm operator with index equal to zero and non-trivial kernel. Also, $\lambda$ belongs to the essential (or \emph{continuous}, in the periodic case) spectrum, $\ess(\cL^\ep)$, provided that either $\cL^\ep - \lambda$ is not Fredholm, or it is Fredholm with non-zero index. Clearly, both the point and essential spectra are subsets of $\sigma(\cL^\ep)$. Moreover, since the operator is closed, $\sigma(\cL^\ep) = \ptsp(\cL^\ep) \cup \ess(\cL^\ep)$ (see \cite{Kat80}, p. 167). The point spectrum comprises isolated eigenvalues with finite (algebraic) multiplicity. The reader is referred to the books by Kato \cite{Kat80} and by Kapitula and Promislow \cite{KaPro13} for further information.
\end{remark}

Since the coefficients of the operator $\cL^\ep$ are periodic, it is well known from Floquet theory that $\cL^\ep$ has no $L^2$-point spectrum and that $\sigma(\cL^\ep) = \ess(\cL^\ep)$ (see Lemma 3.3 in \cite{JMMP14}, or Lemma 59, p. 1487, in \cite{DunSch2}). Moreover, it is possible to parametrize the essential spectrum in terms of Floquet multipliers of the form $e^{i\theta} \in \bbS^1$, $\theta \in \R$ (mod $2\pi$). More precisely, $\lambda \in \sigma(\cL^\ep)$ if and only if
\begin{equation}
\label{Flosp}
\lambda \in \bigcup_{- \pi < \theta \leq \pi} \sigma_\theta,
\end{equation}
where, for each $\theta \in (-\pi,\pi]$, $\sigma_\theta$ denotes the set of complex values, $\lambda \in \C$, for which there exists a bounded non-trivial solution $\bw \in L^\infty (\R) \times L^\infty(\R)$ to the quasi-periodic problem
\begin{equation}
\label{quasip}
\begin{aligned}
\lambda \bw &= \cL^\ep \bw,\\
\bw(T_\ep) &= e^{i \theta} \bw(0).
\end{aligned}
\end{equation}

The union in the right hand side of \eqref{Flosp} is called the \emph{Floquet spectrum} of the periodic coefficient linearized operator $\cL^\ep$ and it coincides with $\sigma(\cL^\ep)$ (cf. \cite{JMMP14,KaPro13}). Thus, the purely essential $L^2$-spectrum can be recast as the union of partial (discrete) spectra $\sigma_\theta$.

One can transform the spectral problem into a family of standard point spectral problems via a \emph{Bloch transformation}. Let us define
\[
y := \frac{\pi x}{T_\ep}, \qquad \bu (y) := e^{-i \theta y/\pi} \bw \Big( \frac{T_\ep y}{\pi}\Big).
\]
Hence, the operator $\partial_x$ transforms into $T_\ep^{-1} (i \theta + \partial_y)$ and the quasi-periodic boundary conditions in \eqref{quasip} become periodic boundary conditions in $y \in [0,\pi]$. The result is the spectral problem,
\[
\lambda \bu = \bB^{-1} \big( T_\ep^{-1} \bA (i \theta + \pi \partial_y) + \bC(T_\ep y/\pi) \big) \bu,
\]
subject to periodic boundary conditions, $\bu(0) = \bu(\pi)$. Let us define,
\[
\bar{\bC}(y) := T_\ep \bC(T_\ep y/\pi) = \begin{pmatrix} T_\ep b(T_\ep y/\pi) & 0 \\ T_\ep a(T_\ep y/\pi) & - T_\ep \end{pmatrix} = \begin{pmatrix} b_1^\ep(y) & 0 \\ a_1^\ep(y) & - T_\ep \end{pmatrix}, 
\]
where the coefficients
\[
a_1^\ep(y) := T_\ep a(T_\ep y/\pi), \qquad b_1^\ep(y) := T_\ep b(T_\ep y/\pi),
\]
are clearly $\pi$-periodic in the $y$ variable. Set $\hat{\lambda} := T_\ep \lambda$ to obtain the equivalent spectral equation,
\begin{equation}
\label{equivsp}
\hat{\lambda} \bu = \cL_\theta^\ep \bu,
\end{equation}
for a family of Bloch operators (indexed by $\theta \in (-\pi,\pi]$) defined as
\begin{equation}
\label{Bloch}
\begin{aligned}
\cL_\theta^\ep &:= \bB^{-1} \big( \bA (i\theta + \pi \partial_y) + \bar{\bC}(y) \big),\\
\cL_\theta^\ep &: \cD(\cL_\theta^\ep) \subset \Ldper([0,\pi]) \times \Ldper([0,\pi]) \rightarrow \Ldper([0,\pi]) \times \Ldper([0,\pi]),
\end{aligned}
\end{equation}
with dense domain $\cD(\cL_\theta^\ep) = \Huper([0,\pi]) \times \Huper([0,\pi])$. The spectrum of each operator in the family, when computed with respect to the space $\Ldper([0,\pi]) \times \Ldper([0,\pi])$, consists entirely of isolated point eigenvalues with finite multiplicity, that is, $
\sigma(\cL_\theta^\ep)_{|\Ldper \times \Ldper} = \ptsp(\cL_\theta^\ep)_{|\Ldper \times \Ldper}$ (see \cite{KaPro13}, pp. 68-70). 

\begin{remark}
It can be shown (cf. \cite{Grd1,KaPro13}) that $\ptsp(\cL^\ep_\theta)_{|\Ldper \times \Ldper}$ depends continuously on the Bloch parameter $\theta$, which is typically a local coordinate for the (continuous) spectrum $\sigma(\cL^\ep)_{|L^2 \times L^2}$, explaining why it consists of \emph{curves} of Floquet spectrum in the complex plane (see Proposition 3.7 in \cite{JMMP14}). This means that $\lambda \in \sigma(\cL^\ep)_{|L^2 \times L^2}$ if and only if $\lambda \in \ptsp(\cL^\ep_\theta)_{\Ldper \times \Ldper}$ for some $\theta \in (-\pi,\pi]$. Consequently, we also have the spectral representation,
\[
\sigma(\cL^\ep)_{|L^2 \times L^2} =  \!\!\bigcup_{-\pi<\theta \leq \pi}\ptsp(\cL^\ep_\theta)_{|\Ldper \times \Ldper}.
\]
\end{remark}

In this fashion, we transform the problem of locating the purely essential spectrum of the linearized operator $\cL^\ep$ into a family of standard point spectral problems for each Bloch operator $\cL_\theta^\ep$ in a suitable periodic space. It is to be observed that $\hat{\lambda} = T_\ep \lambda$, $T_\ep > 0$, yields $\Re \lambda = \Re \hat{\lambda}$ for all $\ep \in (0,\ep_1)$, and, hence, we can state the following 
\begin{proposition}[spectral instability criterion]
\label{propcrit}
For each $\ep \in (0,\ep_1)$ the periodic wave $(U^\ep, V^\ep)$ is spectrally unstable if and only if there exists $\theta_0 \in (-\pi, \pi]$ for which
\[
\ptsp(\cL_{\theta_0}^\ep)_{|\Ldper \times \Ldper} \cap \{ \lambda \in \C \, : \, \Re \lambda > 0\} \neq \varnothing.
\]
\end{proposition}

\begin{remark}[the periodic Evans function]
\label{remEv}
Following \cite{AGJ90} (see also \cite{San02,KaPro13}), one may recast the spectral problem \eqref{equivsp} as a first order system, namely,
\begin{equation}
\label{firstordersyst}
\bu_y = \bD(y,\hat{\lambda},\theta) \bu,
\end{equation}
for $\bu \in \Huper([0,\pi]) \times \Huper([0,\pi])$ and with coefficients
\[
\bD(y,\hat{\lambda},\theta) = \frac{1}{\pi}\bA^{-1} \big( \hat{\lambda} \bB - i \theta \bA - \bar{\bC}(y) \big),
\]
which are analytic in $\hat{\lambda} \in \C$ and periodic functions of $y \in [0,\pi]$ of class $C^1([0,\pi])$. (Notice from Theorem \ref{theoexist} that the subcharacteristic condition holds for the wave speeds, $c(\ep)^2 \tau < 1$ and, therefore, $\bA$ is invertible.) Since each coefficient matrix $\bD(y,\hat{\lambda},\theta)$ is periodic in $y$ we may apply Floquet theory. Let $\bF = \bF(y,\theta,\hat{\lambda})$ denote the principal fundamental matrix for system \eqref{firstordersyst}, that is, the unique solution to $\partial_y \bF= \bD(y,\hat{\lambda},\theta) \bF$ with initial condition $\bF(0,\hat{\lambda},\theta) = \bI$ for every $\hat{\lambda} \in \C$, $\theta \in (-\pi,\pi]$. Hence, $\bar{\bF}( \hat{\lambda}, \theta) :=\bF(\pi, \hat{\lambda}, \theta)$ is the \emph{monodromy matrix} for system \eqref{firstordersyst} and it is an entire function of $\hat{\lambda }\in \C$ (see, e.g., \cite{JMMP14,KaPro13}). Gardner \cite{Grd1} defines the \emph{periodic Evans function} as
\begin{equation}
\label{Evfn}
D(\hat{\lambda},\theta) := \det (\bar{\bF}( \hat{\lambda}, \theta) - e^{i\theta} \bI).
\end{equation}
For each $\theta \in (-\pi,\pi]$ the periodic Evans function is an entire function of $\hat{\lambda} \in \C$, whose isolated zeroes are discrete and coincide in order (multiplicity) and location with the discrete Bloch spectrum, $\ptsp(\cL^\ep_\theta)_{|\Ldper \times \Ldper}$ (see Kapitula and Promislow \cite{KaPro13} for more information).
\end{remark}

\subsection{Relatively bounded perturbations}

Following previous analyses \cite{AlPl21,KDT19,CheDu}, we recast the spectral equation \eqref{equivsp} for each Bloch operator as a perturbation problem. Recall (from Theorem \ref{theoexist}) that the periodic waves have fundamental period and amplitude given by
%
\[
T_\ep = \frac{2\pi}{\omega_0} + O(\ep) =: T_0 + O(\ep),
\]
and
\[
|U^\ep|, |V^\ep| = O(\sqrt{|c(\ep)- c_0|}) = O(\sqrt{\ep}),
\]
respectively, as $\ep \to 0^+$. Hence, we write
\[
T_\ep = T_0 + \sqrt{\ep} T_1,
\]
where $T_1 := (T_\ep - T_0)/\sqrt{\ep} = O(\sqrt{\ep}) = O(1)$ as $\ep \to 0^+$. Same argument applies to the wave speed,
\[
c = c(\ep) = c_0 + \sqrt{\ep} c_1,
\]
where $c_1 := (c(\ep) - c_0)/\sqrt{\ep} = O(1)$. The coefficients are rewritten as
\[
\begin{aligned}
a_1^\ep(y) &= \big(T_0 + \sqrt{\ep} T_1\big) f'(U^\ep(T_\ep y/\pi))\\
&= \big(T_0 + \sqrt{\ep} T_1\big) \big( f'(0) + O(\sqrt{\ep}) \big)\\
&= T_0 f'(0) + O(\sqrt{\ep}).
\end{aligned}
\]
We denote,
\[
a_1^0 := T_0 f'(0), \qquad a_1^1(y) := \frac{a_1^\ep(y) - T_0 f'(0)}{\sqrt{\ep}} = O(1),
\]
so that
\[
a_1^\ep(y) = a_1^0 + \sqrt{\ep} a_1^1(y).
\]
Likewise, since $g'(U^\ep(T_\ep y/\pi)) = g'(0) + O(\sqrt{\ep})$, we obtain
\[
b_1^\ep(y) = b_1^0 + \sqrt{\ep} b_1^1(y),
\] 
where
\[
b_1^0 := T_0 g'(0) > 0, \qquad b_1^1(y) := \frac{b_1^\ep(y) - T_0 g'(0)}{\sqrt{\ep}} = O(1).
\]
Moreover, the coefficient matrices are also recast as
\[
\bA = \begin{pmatrix} c(\ep) & -1 \\ -1 & c(\ep)\tau
\end{pmatrix} = \bA_0 + \sqrt{\ep} \bA_1,
\]
with
\[
\bA_0 := \begin{pmatrix} c_0 & -1 \\ -1 & c_0\tau
\end{pmatrix} = O(1), \qquad \bA_1 := \begin{pmatrix} c_1 & 0 \\ 0 & c_1 \tau
\end{pmatrix} = O(1).
\]
Notice that we can write $\bA = \bA_0 + \sqrt{\ep} c_1 \bB$ so that $\bB^{-1} \bA = \bB^{-1} \bA_0 + \sqrt{\ep} c_1 \bI$. We also define
\[
\bar{\bC}_0 := \begin{pmatrix} b_1^0 & 0 \\ a_1^0 & - T_0 \end{pmatrix} = O(1), \qquad \bar{\bC}_1(y) := \begin{pmatrix} b_1^1(y) & 0 \\ a_1^1(y) & - T_1 \end{pmatrix} = O(1), 
\]
 as $\ep \to 0^+$, so that $\bar{\bC}(y) = \bar{\bC}_0 + \sqrt{\ep} \bar{\bC}_1(y)$. Finally, define
 \[
 \vep := \sqrt{\ep} > 0.
 \]
To sum up, we reformulate the spectral problem \eqref{equivsp} as
\begin{equation}
\label{pertsp}
\hat{\lambda} \bu = \cL_\theta^0 \bu + \vep \cL_\theta^1 \bu,
\end{equation}
where the operators
\begin{equation}
\label{L01}
\begin{aligned}
\cL_\theta^0 &:= \bB^{-1} \bA_0 \big( i \theta + \pi \partial_y \big) + \bB^{-1} \bar{\bC}_0,\\
\cL_\theta^1 &:= c_1 \big( i \theta + \pi \partial_y \big) + \bB^{-1} \bar{\bC}_1(y),
\end{aligned}
\end{equation}
for any $\theta \in (-\pi,\pi]$, are closed and densely defined in $\Ldper([0,\pi]) \times \Ldper([0,\pi])$ with common domain $\cD := \cD(\cL_\theta^j) = \Huper([0,\pi]) \times \Huper([0,\pi])$, $j = 0,1$.

At this point, let us recall that if $\cA, \cB : X \to Y$ are linear operators in $X$, $Y$, Banach spaces, we say that $\cA$ is relatively bounded with respect to $\cB$ (or simply $\cB$-bounded) provided that $\cD(\cB) \subset \cD(\cA)$ and that there exist constants $\alpha, \beta \geq 0$ such that
\[
\| \cA u \| \leq \alpha \| u \| + \beta \| \cB u \|,
\]
for all $u \in \cD(\cB)$ (see Kato \cite{Kat80}, p. 190). The following result plays a key role in the analysis.

\begin{lemma}
\label{lemuno}
For any fixed $\theta \in (-\pi, \pi]$, the operator $\cL_\theta^1$ is relatively bounded with respect to $\cL_\theta^0$.
\end{lemma}
\begin{proof}
Since both operators have the same domain, we only need to show that there exist constants $\alpha, \beta \geq 0$ such that
\[
\| \cL_\theta^1 \bu \|_{\Ldper \times \Ldper} \leq \alpha \| \bu \|_{\Ldper \times \Ldper} + \beta \| \cL_\theta^0 \bu \|_{\Ldper \times \Ldper},
\]
for all $\bu \in \Huper([0,\pi]) \times \Huper([0,\pi])$.

First, notice that since $\bar{\bC}_1(y) = O(1)$ as $\vep = \sqrt{\ep} \to 0^+$ and $|\theta| \leq \pi$ we therefore obtain
\begin{align}
\| \cL_\theta^1 \bu \|_{\Ldper \times \Ldper}  &= \| c_1 \big( i \theta + \pi \partial_y \big) \bu + \bB^{-1} \bar{\bC}_1(y) \bu \|_{\Ldper \times \Ldper} \nonumber\\
&\leq \pi |c_1| \| \partial_y \bu \|_{\Ldper \times \Ldper} + |\theta| |c_1| \| \bu \|_{\Ldper \times \Ldper} + \| \bB^{-1} \bar{\bC}_1(\cdot) \|_{L^\infty} \| \bu \|_{\Ldper \times \Ldper} \nonumber\\
&\leq 2 \pi C_1 \| \partial_y \bu \|_{\Ldper \times \Ldper} + (2 \pi C_1 + C_2) \| \bu  \|_{\Ldper \times \Ldper}, \label{firstest}
\end{align}
for some uniform constants $C_1, C_2 > 0$ which may depend on $\tau$ but are independent of $\bu$ and $\vep$. On the other hand, we estimate
\[
\begin{aligned}
\| \cL_\theta^0 \bu \|_{\Ldper \times \Ldper} &= \| \bB^{-1} \bA_0 (i \theta + \pi \partial_y) \bu +  \bB^{-1}  \bar{\bC}_0 \bu \|_{\Ldper \times \Ldper} \\
&\geq \pi \| \bB^{-1} \bA_0 \partial_y \bu \|_{\Ldper \times \Ldper} - \| \bB^{-1}(\bar{\bC}_0 + i \theta \bA_0) \bu )\|_{\Ldper \times \Ldper},
\end{aligned}
\]
which yields
\begin{equation}
\label{star}
\begin{aligned}
\| \bB^{-1} \bA_0 \partial_y \bu \|_{\Ldper \times \Ldper} &\leq \frac{1}{\pi} \| \cL_\theta^0 \bu \|_{\Ldper \times \Ldper} + \frac{1}{\pi} \| \bB^{-1}(\bar{\bC}_0 + i \theta \bA_0) \bu )\|_{\Ldper \times \Ldper}\\
&\leq \frac{1}{\pi} \| \cL_\theta^0 \bu \|_{\Ldper \times \Ldper}  + \frac{C_3}{\pi} \| \bu \|_{\Ldper \times \Ldper},
\end{aligned}
\end{equation}
for some uniform $C_3 = C_3(\tau) > 0$, inasmuch as $|\bB^{-1}|, |\bar{\bC}_0|$ and $|\bA_0|$ are bounded and $|\theta| \leq \pi$.

Now we claim that there exists a uniform positive constant $C_4 = C_4(\tau) > 0$ such that
\begin{equation}
\label{estC4}
\| \bB^{-1} \bA_0 \partial_y \bu \|_{\Ldper \times \Ldper} \geq C_4 \| \partial_y \bu \|_{\Ldper \times \Ldper},
\end{equation}
for all $\bu \in \Huper([0,\pi]) \times \Huper([0,\pi])$. Indeed, let us denote
\[
\bM_0 = \bM_0(\tau) := \bB^{-1} \bA_0 = \begin{pmatrix} 1 & 0 \\ 0 & 1/\tau \end{pmatrix} \begin{pmatrix} c_0 & -1 \\ -1 & \tau c_0 \end{pmatrix} =  \begin{pmatrix} c_0 & -1 \\ -1/\tau & c_0 \end{pmatrix}.
\]
In view of the subcharacteristic condition ($c_0^2 \tau < 1$ for all $\tau \in (0, \btau)$), it is clear that $r_0(\tau) := \det \bM_0 = c_0^2 - 1/\tau < 0$. Since the matrix
\[
\bM_0^\top \bM_0 = \begin{pmatrix} c_0^2 + 1/\tau^2 & -c_0(1 + 1/\tau) \\ -c_0(1 + 1/\tau) & 1 + c_0^2 \end{pmatrix},
\]
is clearly symmetric, $\det (\bM_0^\top \bM_0) = (\det \bM_0)^2 = r_0(\tau)^2 > 0$ and $\tr (\bM_0^\top \bM_0) = 2c_0^2 + 1 + 1/\tau^2 > 0$, we conclude that $\bM_0^\top \bM_0$ is positive definite and for each fixed $\tau \in (0,\btau)$ there exists a constant $C_4 = C_4(\tau) > 0$ such that
\[
| \bB^{-1} \bA_0 \bw |^2 = \bw^* \bM_0^\top \bM_0 \bw \geq C_4(\tau)^2 |\bw|^2,
\]
for all $\bw \in \C^2$. This yields, in turn, estimate \eqref{estC4} upon integration of the inequality $|\bB^{-1} \bA_0 \partial_y \bu|^2 \geq C_4(\tau)^2 |\partial_y \bu|^2$ over a period.

Therefore, substitution of \eqref{estC4} into \eqref{star} yields
\[
\| \partial_y \bu \|_{\Ldper \times \Ldper} \leq \frac{C_4^{-1}}{\pi} \| \cL_\theta^0 \bu \|_{\Ldper \times \Ldper} + \frac{C_3 C_4^{-1}}{\pi} \| \bu \|_{\Ldper \times \Ldper}.
\]
This implies, together with estimate \eqref{firstest}, that
\[
\| \cL_\theta^1 \bu \|_{\Ldper \times \Ldper} \leq \beta(\tau) \| \cL_\theta^0 \bu \|_{\Ldper \times \Ldper} + \alpha(\tau) \| \bu \|_{\Ldper \times \Ldper},
\]
with
\[
\begin{aligned}
\alpha(\tau) &= 2C_1 C_3 C_4^{-1} + 2\pi C_1 + C_2 > 0,\\
\beta(\tau) &= 2C_1 C_4^{-1} > 0.
\end{aligned}
\]
The lemma is proved.
\end{proof}

Let us now examine the case of a Floquet exponent with $\theta = 0$. The following lemma warrants the existence of an unstable eigenvalue for the unperturbed problem.
\begin{lemma}
\label{lemdos}
$\lambda_0 := b_1^0 = T_0 g'(0) > 0$ is an isolated, simple (or non-degenerate) eigenvalue of the Bloch operator with $\theta = 0$,
\[
\cL_0^0 = \pi \bB^{-1} \bA_0 \partial_y + \bB^{-1} \bar{\bC}_0,
\]
in $\Ldper([0,\pi]) \times \Ldper([0,\pi])$, associated to the constant eigenfunction 
\[
\bu_0 = \left(1, \Big(b_1^0 + \frac{T_0}{\tau}\Big)^{-1} \frac{a_1^0}{\tau} \right)^\top \in \Huper([0,\pi]) \times \Huper([0,\pi]).
\] 
\end{lemma}
\begin{proof}
If we consider a constant eigenfunction of the form $\bu_0 = (u_1, u_2)^\top \in \Huper([0,\pi]) \times \Huper([0,\pi])$ then by direct computation we find that
\[
\cL_0^0 \bu_0 = \bB^{-1} \bar{\bC}_0 \bu_0 = \begin{pmatrix} 1 & 0 \\ 0 & 1/\tau \end{pmatrix} \begin{pmatrix}  b_1^0 & 0 \\ a_1^0 & - T_0 \end{pmatrix} \begin{pmatrix} u_1 \\ u_2\end{pmatrix} = \begin{pmatrix} b_1^0 u_1 \\ \tau^{-1} ( a_1^0 u_1 - T_0 u_2 ) \end{pmatrix} = \lambda_0 \begin{pmatrix} u_1 \\ u_2 \end{pmatrix}.
\]
Since setting $u_1 = 0$ leads to a contradiction, we define $u_1 = 1$ so that $\lambda_0 = b_1^0 > 0$ and
\[
u_2 = \Big(b_1^0 + \frac{T_0}{\tau}\Big)^{-1} \frac{a_1^0}{\tau}.
\]
Notice that $b_1^0 + T_0/\tau > 0$. This shows that $\lambda_0 = b_1^0 = T_0 g'(0) > 0$ belongs to the point spectrum of $\cL_0^0$ and the eigenfunction can be chosen as claimed. This also implies that $\lambda_0$ is an isolated eigenvalue of finite (algebraic) multiplicity.

In order to show that $\lambda_0 = b_1^0$ is a simple eigenvalue, we distinguish between two cases:
\begin{itemize}
\item[(i)] $a_1^0 = c_0 (b_1^0 \tau + T_0)$,
\item[(ii)] $a_1^0 \neq c_0 (b_1^0 \tau + T_0)$.
\end{itemize}

First, we study case (i). After substituting the expressions for $c_0$, $a_1^0$ and $b_1^0$ into $a_1^0 = c_0 (b_1^0 \tau + T_0)$, and because of $T_0 > 0$, we arrive at the relation
\[
f'(0) (1- \tau g'(0)) = f'(0) (1 + \tau g'(0)).
\]
Under our assumptions, $1- \tau g'(0) \neq 1 + \tau g'(0)$ for all $\tau \in (0, \btau)$. Therefore, this last relation implies that case (i) happens only when $f'(0) = 0$. Therefore, $c_0 = a_1^0 = 0$ and the eigenfunction associated to $\lambda_0$ is $\bu_0 = (1,0)^\top$. Let us now suppose that $\bw = (w_1, w_2)^\top \in \Huper([0,\pi]) \times \Huper([0,\pi])$ is a non-trivial first element of a Jordan chain for $\cL_0^0 - \lambda_0$, that is, a non-zero solution to $(\cL_0^0 - \lambda_0) \bw = \bu_0$. In extenso,
\[
(\cL_0^0 - \lambda_0) \bw = \pi \begin{pmatrix} - \partial_y w_2 + b_1^0 w_1 \\ - (\partial_y w_1 + T_0 w_2)/ \tau \end{pmatrix} - b_1^0 \begin{pmatrix} w_1 \\ w_2 \end{pmatrix}
 = \begin{pmatrix} 1 \\ 0\end{pmatrix},
\]
where we have substituted $c_0 = a_1^0 = 0$. Thus, we obtain the system
\[
\begin{aligned}
- \pi \partial_y w_2 &= 1, \\
- \frac{\pi}{\tau} \partial_y w_1 - \Big(b_1^0 + \frac{T_0}{\tau}\Big) w_2 &=0,
\end{aligned}
\]
which has no solution for $w_j \in \Huper([0,\pi])$, $j=1,2$, as it is easily verified.  This shows that there are no non-trivial Jordan chains for $\cL_0^0 - \lambda_0$ and the algebraic multiplicity of $\lambda_0 = b_1^0$ is equal to one.

Let us consider now case (ii). In view of Remark \ref{remEv}, the eigenvalue problem $\cL_0^0 \bu = \lambda \bu$ is equivalent to the first order system
\begin{equation}
\label{firstos}
\partial_y \bu = \bD(\lambda) \bu,
\end{equation}
for $\bu \in \Huper([0,\pi]) \times \Huper([0,\pi])$, where 
\[
\bD(\lambda) = \frac{1}{\pi} \bA_0^{-1} \big(\lambda \bB - \bar{\bC}_0\big) = \frac{1}{\pi} (1 - c_0^2\tau)^{-1} \begin{pmatrix} c_0 \tau (b_1^0 - \lambda) + a_1^0 & - (\lambda \tau + T_0) \\ b_1^0 - \lambda + c_0 a_1^0 & - c_0 ( \lambda \tau + T_0) \end{pmatrix},
\]
as it is easily verified. Notice that $\bA_0$ is invertible thanks to the subcharacteristic condition, $c_0^2\tau < 1$. The eigenvalues of $\bD(\lambda)$ are the $\zeta$-roots of
\[
\begin{aligned}
Q(\zeta,\lambda) &= \big( \zeta + c_0(\lambda \tau + T_0) \big) \big( \zeta - c_0\tau (b_1^0 - \lambda) - a_1^0 \big) + \big( \lambda \tau + T_0 \big) \big( b_1^0 - \lambda + c_0 a_1^0 \big)  \\
&= \zeta^2 + q(\lambda) \zeta + p(\lambda) = 0,
\end{aligned}
\]
where
\[
\begin{aligned}
q(\lambda) &:= c_0(\lambda \tau + T_0) - c_0\tau (b_1^0 - \lambda) - a_1^0,\\
p(\lambda) &:= (1 - c_0^2\tau)(b_1^0 - \lambda) (\lambda \tau + T_0).
\end{aligned}
\]
Thus, the eigenvalues of $\bD(\lambda)$, denoted as $\zeta_j = \zeta_j(\lambda)$, $j = 1,2$, for all $\lambda \in \C$, are given by
\[
\begin{aligned}
\zeta_1(\lambda) &= - \tfrac{1}{2}q(\lambda) - \tfrac{1}{2} \big( q(\lambda)^2 - 4 p(\lambda) \big)^{1/2},\\
\zeta_2(\lambda) &= - \tfrac{1}{2}q(\lambda) + \tfrac{1}{2} \big( q(\lambda)^2 - 4 p(\lambda) \big)^{1/2}.
\end{aligned}
\]
In particular, when $\lambda = b_1^0 > 0$ we have $q(b_1^0) = c_0(b_1^0 \tau + T_0) - a_1^0 \neq 0$ (case (ii)) and $p(b_1^0) = 0$. Consequently,
\[
\zeta_1(b_1^0) = - \tfrac{1}{2}q(b_1^0) - \tfrac{1}{2}|q(b_1^0)|, \qquad \zeta_2(b_1^0) = - \tfrac{1}{2}q(b_1^0) + \tfrac{1}{2}|q(b_1^0)|.
\]
Depending on the sign of $q(b_1^0)$ one of the two eigenvalues is zero. Without loss of generality assume that $q(b_1^0) > 0$. Thus,
\[
\zeta_1(b_1^0) = - q(b_1^0) < 0, \qquad \zeta_2(b_1^0) = 0.
\]
Let us compute
\[
\partial_\lambda q(\lambda) = 2 c_0 \tau, \qquad \partial_\lambda p(\lambda) = (1-c_0^2 \tau)(\tau b_1^0 - T_0 - 2 \tau \lambda),
\]
for all $\lambda \in \C$. This yields, in particular,
\[
\partial_\lambda q(b_1^0) = 2 c_0 \tau, \qquad \partial_\lambda p(b_1^0) = - (1-c_0^2 \tau)(b_1^0 \tau + T_0) < 0.
\]
We also compute
\[
\begin{aligned}
\partial_\lambda \zeta_1(\lambda) &= - \tfrac{1}{2}\partial_\lambda q(\lambda) -  \frac{\big(q(\lambda) \partial_\lambda q(\lambda) - 2 \partial_\lambda p(\lambda)\big)}{2 \big( q(\lambda)^2 - 4 p(\lambda) \big)^{1/2}},\\
\partial_\lambda \zeta_2(\lambda) &= - \tfrac{1}{2}\partial_\lambda q(\lambda) + \frac{\big(q(\lambda) \partial_\lambda q(\lambda) - 2 \partial_\lambda p(\lambda)\big)}{2 \big( q(\lambda)^2 - 4 p(\lambda) \big)^{1/2}}
\end{aligned}
\]
Since $p(b_1^0) = 0$, after straightforward algebra one obtains
\[
\partial_\lambda \zeta_1(b_1^0) \in \R, \qquad \partial_\lambda \zeta_2(b_1^0) = \frac{(1-c_0^2 \tau)(b_1^0 \tau + T_0)}{q_1(b_1^0)} > 0.
\] 

If we interpret the first order (constant coefficient) system \eqref{firstos} as a periodic coefficient problem with period $\pi$, the monodromy matrix is simply $\exp ( \pi \bD(\lambda))$ and the associated Evans function \eqref{Evfn} to the spectral problem for $\cL_0^0 - \lambda$ with $\theta = 0$ is 
\[
D(\lambda,0) = \det \big(\!\exp(\pi \bD(\lambda) ) - \bI \big) = \big( e^{\pi \zeta_1(\lambda)} - 1\big) \big( e^{\pi \zeta_2(\lambda)} - 1\big), \qquad \lambda \in \C.
\]
Clearly, $D(b_1^0,0) = 0$, inasmuch as $\zeta_2(b_1^0) = 0$ (i.e. $\lambda_0 = b_1^0 \in \ptsp(\cL_0^0)$). Now, since $\partial_\lambda \zeta_2(b_1^0) \neq 0$ and $\zeta_1(b_1^0) \neq 0$, we obtain
\[
\partial_\lambda D(b_1^0,0) = \pi \partial_\lambda \zeta_2(b_1^0) \big( e^{\pi \zeta_1(b_1^0)} - 1\big) \neq 0.
\]
From Lemma 8.4.1 in \cite{KaPro13} we conclude that $\lambda_0 = b_1^0 > 0$ has algebraic multiplicity equal to one. The lemma is now proved.
\end{proof}

We now pay attention to the perturbed spectral problem \eqref{pertsp} for the Bloch operator with $\theta = 0$, namely, to
\begin{equation}
\label{pertsp2}
\hat{\lambda} \bu = \cL_0^0 \bu + \vep \cL_0^1 \bu, \quad \qquad \bu \in \Huper([0,\pi]) \times \Huper([0,\pi]).
\end{equation}
Let us complexify the family and define
\begin{equation}
\label{compf}
\begin{aligned}
\widetilde{\cL}_0(\kappa) &:= \cL_0^0 + \kappa \cL_0^1,\\
\widetilde{\cL}_0(\kappa) &: \cD(\widetilde{\cL}_0(\kappa)) \subset \Ldper([0,\pi]) \times \Ldper([0,\pi]) \to \Ldper([0,\pi]) \times \Ldper([0,\pi]),
\end{aligned}
\end{equation}
where $\kappa \in B_{\hat{\vep}} = \{\kappa \in \C \, : \, |\kappa| < \hat{\vep}\}$, for some $\hat{\vep} > 0$ to be determined. That is, we extend the perturbed problem \eqref{pertsp2} to a complex open neighborhood of the origin. The family of operators in \eqref{compf} has a common domain, $\cD(\widetilde{\cL}_0(\kappa)) = \cD = \Huper([0,\pi]) \times \Huper([0,\pi])$, independent of $\kappa \in B_{\hat{\vep}}$. Let us recall a concept from the perturbation theory of linear operators (see Kato \cite{Kat80}, p. 375, or Hislop and Sigal \cite{HiSi96}, \S 15.4).

\begin{definition}
A family of linear closed operators $\cL(\kappa) : X \to Y$, with $X, Y$ Banach spaces, defined for $\kappa$ in an open complex domain $B \subset \C$ containing the origin, is said to be \emph{holomorphic of type (A)} if the domains are independent of $\kappa$, that is, $\cD(\cL(\kappa)) \equiv \cD$ for all $\kappa \in B$, and the mapping $\kappa \mapsto \cL(\kappa) \bu$ is holomorphic in $\kappa \in B$ for each $\bu \in \cD$. In that case, $\cL(\kappa)$ has a convergent Taylor expansion of the form
\[
\cL(\kappa) \bu = \cL^{(0)} \bu + \kappa \cL^{(1)} \bu + \kappa^2 \cL^{(2)} \bu + \ldots,
\]
converging in a disk $|\kappa| < r$ inside $B$ independent of $\bu$. Here $\cL(0) = \cL^{(0)}$ and $\cL^{(j)} : \cD \subset X \to Y$ are linear operators for each $j \in \N$.
\end{definition}

An important consequence of Lemma \ref{lemuno} is the following
\begin{lemma}
\label{lemtres}
There exists $\hat{\vep} > 0$ such that the family of operators \eqref{compf} is a holomorphic family of type (A) for $\kappa \in B_{\hat{\vep}}$.
\end{lemma}
\begin{proof}
Since the operator $\cL_0^1$ is relatively bounded with respect to $\cL_0^0$ on the space $\Ldper([0,\pi]) \times \Ldper([0,\pi])$ (see Lemma \ref{lemuno}), then there exist $\alpha, \beta > 0$ such that
\[
\| \cL_0^1 \bu \|_{\Ldper \times \Ldper} \leq \alpha \| \bu \|_{\Ldper \times \Ldper} + \beta \| \cL_0^0 \bu \|_{\Ldper \times \Ldper},
\]
for all $\bu \in \cD = \Huper([0,\pi]) \times \Huper([0,\pi])$. Hence, we apply Theorem VII-2.6 and Remark VII-2.7 in \cite{Kat80}, pp. 377-378, to conclude that there exists a radius $\hat{\vep} > 0$ such that the family $\widetilde{\cL}_0(\kappa) = \cL_0^0 + \kappa \cL_0^1$ is holomorphic of type (A) for $|\kappa| < \hat{\vep}$ (actually, $\hat{\vep} < \beta^{-1}$). Notice that the operators $\cL_0^j$, $j=0,1$, are closed and have common domain $\cD$.
\end{proof}

\begin{remark}
In contrast with the viscous case \cite{AlPl21}, here the unperturbed operator $\cL_0^0$ is not self-adjoint and we rely on the property of \eqref{compf} of being an holomorphic family of type (A) and on the simplicity of the unstable eigenvalue (Lemma \ref{lemdos}) in order to conclude the analyticity of the eigenvalues of the perturbed problem as we shall see below.
\end{remark}

\begin{lemma}
\label{lemcuatro}
There exists $\vep_2 > 0$ and an analytic family of discrete, non-degenerate (i.e. simple) eigenvalues $\lambda(\kappa)$ of the operators $\widetilde{\cL}_0(\kappa) = \cL_0^0 + \kappa \cL_0^1$ for $|\kappa| < \vep_2$ such that $\lambda(0) = \lambda_0 = b_1^0$. Moreover, for each $0 < \vep < \vep_2$ there holds
\begin{equation}
\label{intersecti}
\ptsp(\cL_0^0 + \kappa \cL_0^1)_{|\Ldper \times \Ldper} \cap \{ \lambda \in \C \, : \, |\lambda - b_1^0| < r(\vep)\} \neq \varnothing,
\end{equation}
for some $0 < r(\vep) = O(\vep)$.
\end{lemma}
\begin{proof}
From Lemmata \ref{lemtres} and \ref{lemdos} we know that $\cL_0^0 + \kappa \cL_0^1$ is an holomorphic family of type (A) for $|\kappa| < \hat{\vep}$ and that $\lambda_0 = b_1^0$ is a discrete, non-degenerate (simple) eigenvalue of $\cL_0^0$. Therefore, we apply Theorem 15.10 in Hislop and Sigal \cite{HiSi96}, p. 155, to conclude that there exists an analytic expansion
\[
\lambda(\kappa) = \lambda_0 + \kappa \lambda_1 + \kappa^2 \lambda_2 + \ldots,
\]
convergent in a (possible smaller) neighborhood of the origin, say $|\kappa| < \vep_2 \leq \hat{\vep}$, such that $\lambda(\kappa)$ is a discrete, simple eigenvalue of $\widetilde{\cL}_0(\kappa)$ with $\lambda(0) = \lambda_0$. Property \eqref{intersecti} follows from continuity.
\end{proof}

We are now ready to prove the main instability result.

\subsection{Proof of Theorem \ref{theoinst}}

Since $\vep = \sqrt{\ep}$, we choose $0 < \ep_2 := \min \{\ep_1, \vep_2^2\}$ and apply Lemma \ref{lemcuatro} to conclude that there exists $\theta_0 := 0 \in (-\pi,\pi]$ for which the spectrum of the perturbed Bloch operator with $\theta = 0$ and $\ep \in (0, \ep_2)$ intersects the unstable half plane (in view that $\lambda_0 = b_1^0 > 0$):
\[
\ptsp(\cL_{0}^\ep)_{|\Ldper \times \Ldper} \cap \{ \lambda \in \C \, : \, \Re \lambda > 0\} \neq \varnothing.
\]
Therefore, from the spectral instability criterion stated in Proposition \ref{propcrit} we obtain the result. Notice that if we vary $\theta \approx 0$ (within a small neighborhood of the origin) we obtain curves of Floquet spectrum that locally remain in the unstable half plane. This completes the proof of Theorem \ref{theoinst}. 
\qed

\section{Concluding remarks}
\label{secconcl}

In this paper, we have proved the existence of bounded periodic waves to a hyperbolic system modeling (non-Fickian) diffusion of Cattaneo-Maxwell type and nonlinear advection of a scalar quantity, together with a source term of logistic type. For that purpose we have showed that, under certain structural assumptions, a family of small-amplitude and finite period waves emerges from a local Hopf bifurcation around a critical value of the wave speed. These waves are subcharacteristic, that is, they satisfy condition \eqref{subchar}. The result is valid for all parameter values of the relaxation time $\tau$ which remain below a threshold value. Although the hyperbolic system \eqref{hypVBL} is designed for small values of $\tau$ (in order to be consistent with physical or biological applications), it is not clear whether our existence result can be extended beyond the subcharacteristic regime. We also presented some numerical approximations of the periodic waves in the case of the hyperbolic Burgers-Fisher model.

In addition, we have studied the stability of the family of periodic waves as solutions to the PDE system. In particular, we proved that these waves are spectrally unstable, or in other words, that the Floquet spectrum of the linearized operator intersects the unstable complex half plane. Heuristically, this instability result can be interpreted as follows. When the small parameter tends to zero, $\ep \to 0$, these small-amplitude waves shrink to the zero equilibrium solution and the linearized operator around the wave tends (formally) to a constant coefficient linearized operator around it. The latter has a spectrum determined by a dispersion relation that intersects the unstable half plane thanks to the positive sign of the reaction term at the rest state ($g'(0) > 0$). We then invoke the classical perturbation theory for linear operators to conclude. This technique has been recently applied to study stability of small-amplitude waves for scalar Hamiltonian equations \cite{KDT19}, scalar viscous balance laws \cite{AlPl21}, and scalar reaction-diffusion equations coupled to one ordinary differential equation \cite{CheDu}. This appears to be a general phenomenon and the method could be used to study the instability of small periodic waves arising in other contexts. From a technical viewpoint, it is to be observed that we proved that the unstable eigenvalue of the unperturbed operator is simple or non-degenerate (see Lemma \ref{lemdos}). This happens in other models as well (see, e.g., \cite{CheDu}). As a consequence of this non-degeneracy there exists an analytic expression for the (also non-degenerate) eigenvalues of the perturbed problem, as the theory guarantees. This non-degeneracy, however, is by no means a necessary condition to conclude instability. (In the Hamiltonian case, for example, pairs of \emph{collided eigenvalues} of the linearized operator at zero amplitude usually appear \cite{KDT19}.) Indeed, in the general case with eigenvalues of higher multiplicity the expansions for eigenvalues may no longer be analytic, but they can be represented by Pusieux series which remain continuous in the limit. Thus, we strongly advocate for perturbation theory as a technique to examine instabilities of small amplitude waves in more general situations. Finally, and for the particular hyperbolic system considered in this paper, we highlight the existence of large period waves arising from a \emph{global} homoclinic bifurcation (such as in the parabolic case \cite{AlPl21}) and the orbital (nonlinear) stability analysis of both families of waves, as open problems which are worthy of further investigations.


\section*{Acknowledgements}

E. \'{A}lvarez gratefully acknowledges the support of the Program ``Apoyos FE\-NO\-MEC 2021''. The work of R. G. Plaza was partially supported by DGAPA-UNAM, program PAPIIT, grant IN-104922.

\def\cprime{$'\!\!$} \def\cprimel{$'\!$}





\end{document}